\documentclass[11pt]{article}

\usepackage{amsmath}
\usepackage{amssymb}

\newtheorem{thm}{Theorem}
\newtheorem{cor}[thm]{Corollary}

\newtheorem{lemma}[thm]{Lemma}

\newtheorem{prop}[thm]{Proposition}

\newenvironment{proof} { \emph{Proof.} } { {\rule{2mm}{2mm}}\\}

\newcommand{\bea}{\begin{eqnarray*}}
\newcommand{\eea}{\end{eqnarray*}}
\newcommand{\hyp}{\mathbb{H}}

\newcommand{\hnn}{\hyp^{n+1}}

\newcommand{\e}{\epsilon}
\newcommand{\al}{a}

\newcommand{\R}{\mathbb{R}}

\newcommand{\B}{\mathcal{B}}
\newcommand{\Z}{\mathbb{Z}}

\newcommand{\F}{{\mathcal{F}}}
\newcommand{\G}{{\mathcal{G}}}
\newcommand{\N}{{\mathcal{N}}}
\newcommand{\K}{{\mathcal{K}}}
\newcommand{\Hcal}{{\mathcal{H}}}
\newcommand{\Vcal}{{\mathcal{V}}}
\newcommand{\g}{{\gamma}}

\newcommand{\n}{{\gamma}}
\newcommand{\m}{{\eta}}
\renewcommand{\k}{{\zeta}}

\newcommand{\geod}{{\xi}}
\newcommand{\point}{{p}}
\newcommand{\Lie}{\mathcal{L}}
\newcommand{\Nie}{\mathcal{U}}

\newcommand{\bm}{\begin{pmatrix}}
\newcommand{\fm}{\end{pmatrix}}

\begin{document}
\title{Large scale geometry of certain solvable groups}
\author{Tullia Dymarz\footnote{Supported in part by NSERC PGS B}}
\maketitle

\begin{abstract}
In this paper we provide the final steps in the proof of quasi-isometric rigidity of a class of non-nilpotent polycyclic groups.
To this end, we prove a rigidity theorem on the boundaries of certain negatively curved homogeneous spaces and combine it with work of Eskin-Fisher-Whyte  and Peng on the structure of quasi-isometries of certain solvable Lie groups.
\end{abstract}

\section{Introduction}

A class of finitely generated groups $\mathcal{C}$ is said to be \emph{quasi-isometrically rigid} if any group quasi-isometric to a group in $\mathcal{C}$ is also in $\mathcal{C}$ (up to extensions of and by finite groups).
In this paper we provide the final steps in the proof of the following theorem. 

\begin{thm}\cite{EFW}\label{polyrigid} Suppose $M$ is a diagonalizable matrix with $\det{M}=1$ and no eigenvalues on the unit circle. Let $G_M=\R \ltimes_M \R^n$. If $\Gamma$ is a finitely generated group quasi-isometric to $G_M$ then $\Gamma$ is virtually a lattice in $\R \ltimes_{M'} \R^n$ where $M'$ is a matrix that has the same absolute Jordan form as $M^\alpha$ 
for some $\alpha \in \R.$
\end{thm}

\noindent{}The absolute Jordan form differs from the usual Jordan form in that complex eigenvalues are replaced by their absolute value. The group $G_M$ is a solvable Lie group and therefore admits only cocompact lattices. 
Therefore, in the language of quasi-isometric rigidity, Theorem \ref{polyrigid} says that $\mathcal{C}_M$, the class of lattices in groups $G_{M'}$ where $M'$ has the same absolute Jordan form as some power of $M$, is quasi-isometrically rigid. \\

\noindent{}Lattices in $G_M$ are examples of \emph{(virtually) polycyclic} groups. By a theorem of Mostow \cite{Mos}, any polycyclic group is virtually a lattice in some solvable Lie group and conversely any lattice in a solvable Lie group is virtually polycyclic. So Theorem \ref{polyrigid} is a step
towards proving a standard conjecture, first officially stated in \cite{EFW}, 
that the class of all polycyclic groups is quasi-isometrically rigid.\\

\noindent{}Theorem \ref{polyrigid} was first conjectured by Farb-Mosher in \cite{FM4} and first announced by Eskin-Fisher-Whyte in \cite{EFW}. 
In the special case when $M$ is a $2 \times 2$ matrix, $G_M$ is the usual three dimensional Sol geometry. Quasi-isometric rigidity of lattices in Sol was the first breakthrough in the study of quasi-isometric rigidity of polycyclic groups. 
The proof can found in \cite{EFW2,EFW1}.\\

\noindent{}The first step in proving Theorem \ref{polyrigid} is the following theorem which was proved in the special case of Sol in \cite{EFW2,EFW1} and in full generality by Peng in \cite{P1,P}.\\

\noindent{}{\bf Theorem} \emph{\cite{EFW} 
Suppose $M$ is a diagonalizable matrix with $\det{M}=1$ and no eigenvalues on the unit circle. Let $G_M=\R \ltimes_M \R^n$ 
and let $h:G_M \to \R$ be projection to the $\R$ coordinate. 
Then every quasi-isometry of $G_M$ permutes level sets of $h$ up to bounded distance. (We call such a map ``height respecting''.)}\\

\noindent{}In the case of Sol, the second step in proving Theorem \ref{polyrigid} is Theorem 3.2 in \cite{FM2}. For the general case, the second steps in proving Theorem \ref{polyrigid} are the main results of this paper (Theorem \ref{mytukia2} and Proposition \ref{thebilipprop} below). Theorem \ref{mytukia2} involves studying maps of $(\R^n, D)$ where  $D$ denotes a ``metric'' on $\R^n\simeq \R^{n_1} \oplus \R^{n_2} \oplus \cdots \oplus \R^{n_r}$ of the form 
$$D(x,y)=\max \{|x_1-y_1|^{1/\alpha_1}, \ldots ,|x_r-y_r|^{1/\alpha_{r}}\}$$
where  $0<\alpha_i < \alpha_{i+1}$, and $x_i,y_i \in \R^{n_i}$. 
We consider $QSim_{D}(\R^n)$, the group of \emph{quasisimilarities} of $\R^n$ with respect to the metric $D$. A quasisimilarity is simply a bilipschitz map $F$ that satisfies
$$ K_1 D(x,y) \leq  D(F(x), F(y)) \leq K_2 D(x,y)$$ 
where we keep track of both $K_1,K_2$.  A \emph{uniform} group of quasisimilarities is one where the ratio $K_2/K_1$ is fixed.  If $K_2/K_1=1$ we call $F$ a \emph{similarity}.
We also consider $ASim_{D}(\R^n)$, the group of \emph{almost similarities} of $\R^n$. An almost similarity is a similarity composed with an \emph{almost translation}, a map of the form
  $$(x_1,x_2, \cdots , x_r) \mapsto (x_1 + B_1(x_2, \cdots , x_r), x_2 + B_2(x_3,\cdots,x_r), \cdots , x_r + B_r)$$
that is also a $QSim_D$ map.
For more information on the maps $B_i$ and other definitions see Section \ref{bilipDM}.\\

\noindent{}The main theorem of this paper is the following:

\begin{thm} \label{mytukia2} Let $\G$ be a uniform separable subgroup of $QSim_{D}(\R^n)$ that acts cocompactly on 
the space of distinct pairs of points of $\R^n$. Then there exists a map $F \in QSim_{D}(\R^n)$ such that 
$$F \G F^{-1} \subset ASim_{D}(\R^n).$$
\end{thm}

\noindent{\bf Geometry.} Although at first sight Theorems \ref{polyrigid} and \ref{mytukia2} seem unrelated, there is a geometric connection. The solvable Lie group $G_M$ has natural foliations by negatively curved homogeneous spaces. These negatively curved homogenous spaces in turn have boundaries which can be identified with the space $(\R^n, D)$. Quasi-isometries of $G_M$ induce $QSim_D$ maps of $(\R^n, D)$ whereas isometries induce $Sim_D$ maps, similarities, of $(\R^n, D)$ . We will discuss this geometric connection in more detail in Section 2.
We will also use this connection in Section \ref{QIgroup} to describe the quasi-isometry group $QI(G_M)$ in terms of $QSim_D$ maps.
\\

\noindent{}In Section \ref{tukiasection} we prove Theorem \ref{mytukia2}. The proof of Theorem \ref{mytukia2} is modeled after Tukia's theorem \cite{T} on quasiconformal maps of $S^n$. Tukia's theorem states that, for $n\geq 2$, any uniform group of quasiconformal maps of $S^n$ that acts cocompactly on the space of distinct triples of $S^n$ can be conjugated by a quasiconformal map into the group of conformal maps of $S^n$.\\

\noindent{\bf Conformal structures.}
The key ingredient in the proof of Tukia's theorem is that quasiconformal maps are almost everywhere differentiable. This allows one to define a measurable conformal structure on $S^n$. For our theorem, new ideas are needed since $QSim_D$ maps are not necessarily differentiable. $QSim_D$ maps do, however, preserve a certain flag of foliations and along the leaves of these foliations, they are differentiable. This allows us to define the notion of a ``$D$-foliated'' conformal structure on $\R^n$ (see Section \ref{moredim}). \\

\noindent{\bf Proof outline.} We prove Theorem \ref{mytukia2} by induction on the number of distinct $\alpha_i$ occuring in $D$. 
The base case, when there is only one distinct $\alpha_i$, is discussed in Section \ref{basecase}.  The induction step is proved in several parts. First, we set up the induction step in Section \ref{inductionstep}. In Section \ref{moredim}, we treat the case 
when the multiplicity of the smallest eigenvalue is greater than one. This case follows the outline of Tukia's proof. When the multiplicity is equal to one, we provide a new proof (see Section \ref{onedim}). The next two parts of the proof of the induction step are treated in \ref{multconst} and \ref{rotation}.\\
 
\noindent{}In Section 4 we finish the proof of Theorem \ref{polyrigid}. Given a finitely generated group $\Gamma$ quasi-isometric to $G_M$ we 
use Theorem \ref{mytukia2} and work of Peng \cite{P}
to get an action of $\Gamma$ on $G_M$ by \emph{almost isometries} (see Section \ref{nonameyet} for details and definitions). The reason the action of $\Gamma$ on $G_M$ is by almost isometries and not by actual isometries is precisely because in Theorem \ref{mytukia2} we are unable to conjugate a uniform group of quasisimilarities into the group of similarities but only into the group of almost similarities. 
If the action were actually by isometries then we would be done (see Section \ref{nonameyet}). 
Instead, in Section \ref{showinpoly}, we use the structure of the almost isometries to prove that $\Gamma$ is polycyclic. The main tool we use is the following proposition:

\begin{prop}\label{thebilipprop}Suppose a group $\mathcal{N}$ quasi-acts properly on $\R^n$ by $K$-$Bilip_{D}$ almost translations. Then $\mathcal{N}$ is finitely generated nilpotent.
\end{prop}

\noindent{}This proposition is also used by Peng in \cite{P} to show that a larger class of polycyclic groups is rigid (see also Theorem \ref{polythm} in Section \ref{showinpoly}). \\

\noindent{}Finally, since $\Gamma$ is polycyclic, then as mentioned earlier, 
$\Gamma$ is virtually a lattice in some solvable Lie group $G$.
In Section \ref{RbyRn}, we finish the proof of Theorem \ref{polyrigid} by showing that $G\simeq \R\ltimes_{M'} \R^n$ where $M'$ is a matrix that has the same absolute Jordan form as $M^{\alpha}$ for some $\alpha \in \R$. \\

\noindent{\bf Acknowledgments.}
I would especially like to thank my advisor Benson Farb along with Alex 
Eskin, David Fisher and Kevin Whyte for giving me the opportunity to work 
on this project and for all of their guidance. I also took great benefits from 
the advice and ideas of Bruce Kleiner, Pierre Pansu, Irine Peng, and Juan 
Souto. I would also like to thank Anne Thomas and Irine Peng
for their input on earlier drafts of this paper.

\section{Preliminaries}
In this section we collect some definitions and preliminary results that will be used in the proofs of Theorem \ref{polyrigid} and Theorem  \ref{mytukia2}. In Section \ref{standarddefinitions} we collect some standard terminology whereas in Section \ref{notation} we introduce new notation. In Section \ref{bilipDM} we prove some facts about the structure of $QSim_{D_M}$ maps that are needed in the proof of Theorem \ref{mytukia2}. In Section \ref{geomofGM} we describe the geometry of the solvable Lie groups defined in Theorem \ref{polyrigid}. Finally, in Section \ref{boundary} we relate $G_M$ and $QSim_{D_M}$ maps.

\subsection{Standard definitions.}\label{standarddefinitions}
The following are standard definitions. For more details see for instance \cite{BH}.\\

\noindent{\bf Quasi-isometry.} A map $\varphi:X \to Y$ between metric spaces is said to be 
 a $(K,C)$ \emph{quasi-isometry} if there exists $K,C$ such that
 $$ -C + 1/K\ d(x,y) \leq d(\varphi(x),\varphi(y)) \leq K\ d(x,y) + C$$
and the $C$ neighborhood of $\varphi(X)$ is all of $Y.$\\

\noindent{\bf Bounded distance.} We say two maps $\varphi,\varphi':X \to Y$ are at a bounded distance from each other if there exists some $C> 0$ such that
$$ \sup_{x\in X}{ d(\varphi(x),\varphi'(x)) } < C.$$
Then we write $d(\varphi,\varphi') < C$ or $d(\varphi,\varphi')< \infty$ if we do not need to specify $C$.\\

\noindent{\bf Coarse inverse.} Every quasi-isometry $\varphi:X \to Y$ has a \emph{coarse inverse} $$\bar{\varphi}: Y \to X$$ which is a quasi-isometry with the property that $d(\bar{\varphi}\circ \varphi, Id_X) < \infty$ and  $d(\varphi\circ \bar{\varphi}, Id_Y) < \infty$.\\

\noindent{\bf Quasi-isometry group.} Given metric space $X$ we define the quasi-isometry group $QI(X)$ to be set of equivalence classes of quasi-isometries $\varphi: X \to X$ where $\varphi$ and $\varphi'$ are identified if $d(\varphi,\varphi') < \infty$. 
Multiplication is given by composition. \\

\noindent{\bf Quasi-action.}
A group $\G$ \emph{quasi-acts} on a metric space $X$ if there exist constants $K,C>0$ and a map $A:\G \times X \to X$ such that:
 \begin{itemize}
\item  $A_G:X \to X$ is a $(K,C)$ quasi-isometry for each $G\in \G$
 \item $d(A_G\circ A_F, A_{GF}) < C$ for all $G,F\in \G$
 \end{itemize}
The quasi-action is said to be 
\emph{cobounded} if there exists a constant $R\geq 0$ such that for each $x\in X$ the $R$-neighborhood of the orbit $\G\cdot x$ is all of $X$. The quasi-action is \emph{proper} if for each $R\geq 0$ there exists a $C'\geq 0$ such that for all $x,y \in X$ the cardinality of the set $\{G \in \G \mid (\G\cdot N(x,R)) \cap N(y,R)\neq \emptyset \}$ is at most $C'$. Here $N(x,R)$ denotes the $R$-neighborhood of $x$.\\

\noindent{\bf Quasi-conjugacy.} If a group $\G$ is endowed with a left-invariant metric and $\varphi: \G \to X$ is a quasi-isometry then we can define a quasi-action of $\G$ on $X$  by setting 
$$A_G=\varphi L_G \bar{\varphi}$$
where $L_G$ denotes left multiplication by $G \in \G$. This quasi-action is cobounded and proper.\\

\noindent{\bf Word metric.} Any finitely generated group $\Gamma$ can be viewed as a metric space by fixing a generating set $S$ and  defining a left invariant word metric as follows:
$$d(\g,\m)=||\g^{-1}\m|| \textrm{ for  all }  \g,\m \in \Gamma$$
where $||\g||$ denotes the minimum number of generators in $S$ required to write $\g$. \\

\noindent{\bf Quasi-isometric rigidity.} 
A class of finitely generated groups $\mathcal{C}$ is said to be \emph{quasi-isometrically rigid}  if whenever a finitely generated group $\Lambda$ is quasi-isometric to $\Gamma \in \mathcal{C}$, then $\Lambda$ is \emph{virtually} in $\mathcal{C}$.
We say a group $\Gamma$ \emph{virtually} has a property $P$ if up to extensions of and by a finite group $\Gamma$ has $P$.\\

\noindent{\bf Bilipschitz, similarity and quasisimilarity.} A map $f:X \to Y$ between metric spaces is said to be 
\begin{itemize}
\item a $K$-\emph{\bf bilipschitz} map if 
$$1/K\ d(x,y) \leq d(f(x),f(y)) \leq Kd(x,y)$$
\item an $N$-\emph{\bf similarity} if 
$$ d(f(x),f(y)) = N\ d(x,y)$$
\item an $(N,K)$-\emph{\bf quasisimilarity} if 
$$N/K\ d(x,y) \leq d(f(x),f(y)) \leq N K\ d(x,y).$$
\end{itemize}

\subsection{The metric $D_M$ and associated maps and groups.}\label{notation}
\noindent{\bf The Metric $\mathbf{D_M}$.} Fix $M$ an $n\times n$ diagonal matrix with real eigenvalues $e^{\alpha_i}$  with $ \alpha_{i+1} > \alpha_i>0$. We will write $(x_1, \ldots , x_r) \in \R^n$, with $x_i \in \R^{n_i}$ where ${n_i}$ is the multiplicity of the eigenvalue $e^{\alpha_i}$. Define the ``metric'' $D_M$ as follows:
$$D_M(x,y)=\max \{|x_1-y_1|^{1/\alpha_1}, \ldots ,|x_r-y_r|^{1/\alpha_r}\}.$$
The map $D_M$ is not quite a metric, as it does not satisfy the triangle inequality, but some power of it does. We write $D_M$ instead of simply $D$ as we did in the introduction because $D_M$ will be connected with the solvable Lie group $G_M$ in section \ref{geomofGM}.
These metrics were also considered by Tyson in \cite{Ty}.\\

\noindent{\bf Special maps.} We call a map $F:\R^n \to \R^n$ a
\begin{itemize}
\item $Bilip_{D_M}$ map if it is bilipschitz with respect to $D_M$,
\item $Sim_{D_M}$ map if it is a similarity with respect to $D_M$, 
\item $QSim_{D_M}$ map if it is a quasisimilarity with respect to $D_M$.
\item $ASim_{D_M}$ map if it is a $Sim_{D_M}$ map composed with a  \emph{almost translation}, i.e. a $Bilip_{D_M}$ map of the form
$$F(x_1,x_2, \ldots , x_r) = (x_1 + B_1(x_2, \ldots , x_r), x_2 + B_2(x_3,\ldots,x_r), \ldots , x_r + B_r).$$
\end{itemize}
We write $K$-$Bilip_{D_M}$, $N$-$Sim_{D_M}$, or $(K,N)$-$QSim_{D_M}$ if we want to keep track of the constants. If $F$ is an $N$-$Sim_{D_M}$ map, we refer to $N$ as the \emph{similarity constant} of $F$.\\ 

\noindent{\bf Special subgroups.} We write $Bilip_{D_M}(\R^n)$ to denote the group of all $Bilip_{D_M}$ maps of $\R^n$. The groups $Sim_{D_M}(\R^n)$, $QSim_{D_M}(\R^n)$ and $ASim_{D_M}(\R^n)$ are defined similarly.\\

\noindent{\bf Uniform subgroups.}
A \emph{uniform} subgroup of 
\begin{itemize}
\item $QSim_{D_M}(\R^n)$ is a group of $(K,N)$-$QSim_{D_M}$ maps where
  $K$ is fixed.
\item $Bilip_{D_M}(\R^n)$ is a group of $K$-$Bilip_{D_M}$ maps where
  $K$ is fixed.
\end{itemize}
Note that $Bilip_{D_M}(\R^n)$ and $QSim_{D_M}(\R^n)$ are
equal as groups but their uniform subgroups are not the same.
 We say that $\G$ acts on $\R^n$ by $QSim_{D_M}$ maps, or that we have a $QSim_{D_M}$ action of $\G$ on $\R^n$, if there is a homomorphism
$$\phi: \G \to QSim_{D_M}(\R^n).$$
Similarily, we can define $Bilip_{D_M}$, $Sim_{D_M}$ and $ASim_{D_M}$ actions.\\

\noindent{}{\bf Standard dilation.}  For $t>0$, the map 
$$\delta_t(x_1,x_2,\ldots, x_r)=(t^{\alpha_1}x_1,t^{\alpha_2}x_2,\ldots, t^{\alpha_r }x_r)$$
satisfies $$ {D_M}(\delta_t(p),\delta_t(q))=t {D_M}(p,q)$$
for all $p,q \in \R^n$.
Therefore $\delta_t \in Sim_{D_M}(\R^n)$. We will call $\delta_t$ the \emph{standard dilation} of $\R^n$ with respect to  $D_M$. \\

%
%
\subsection{Properties of  $QSim_{D_M}$ maps.}\label{bilipDM}
In this section we examine the structure of $QSim_{D_M}$ maps. Recall that a  $QSim_{D_M}$ map is simply a $Bilip_{D_M}$ maps composed with $Sim_{D_M}$. Proposition \ref{foliationlemma} can also be found in \cite{Ty}, Section 15.
 
\begin{prop}\label{foliationlemma} A $Bilip_{D_M}$ map has the form 
$$ (x_1,x_2, \cdots, x_r) \mapsto (f_1(x_1,\cdots, x_r), \cdots, f_r( x_r ))$$
where $f_i(x_i,\cdots, x_r)$ is bilipschitz as a function of $x_i$ with respect to the standard metric on $\R^{n_i}$ and, for $l>i$, is H\"older continuous in $x_l$, with exponent $\alpha_i/\alpha_l$.
\end{prop}
\begin{proof} 
For two points $p,q \in \R^n$ define $\triangle_\beta(p,q)\geq 0$ to be the infimum over all finite sequences $\{ p_j\}_{j=0}^m$ where  $p=p_0$ and $q=p_m$ of
$$ \sum_{j=1}^{m} [D_M(p_{j-1},p_j)]^\beta.$$
For each $1\leq i \leq r$ define
\bea
D_i^-(x,x')&=&\max \{|x_1-x'_1|^{1/\alpha_1}, \cdots ,|x_{i-1}-x'_{i-1}|^{1/\alpha_{i-1}}\}, and\\
D_i^+(x,x')&=&\max \{|x_i-x'_i|^{1/\alpha_i}, \cdots ,|x_r-x'_r|^{1/\alpha_r}\},
\eea
so that we can write
$$D_M(x,x')=\max\{ D_i^-(x,x'),D_i^+(x,x')\}.$$
%
%

\begin{figure}[htbp]
\begin{center}

\setlength{\unitlength}{.5in} 
\begin{picture}(10,5)(0,0) 
\linethickness{1pt} 

\put(1,1){\line(1,0){3}} 
\put(1,1){\circle*{.15}}
\put(4,1){\line(0,1){2}}
\put(4,1){\circle*{.15}}
\put(4,3){\line(1,0){2}} 
\put(4,3){\circle*{.15}}
\put(6,3){\line(0,1){1}}
\put(6,3){\circle*{.15}}
 \put(6,4){\circle*{.15}}

\put(1.25,1.25){\makebox(0,0){$\frac{1}{k}$}} 
\put(1.75,1.25){\makebox(0,0){$\frac{1}{k}$}} 
\put(2.5,1.25){\makebox(0,0){$\cdots$}} 
\put(3.75,1.25){\makebox(0,0){$\frac{1}{k}$}} 

\put(7.65,1.8){\makebox(0,0){$\displaystyle \Delta_3 (p_0,p^{*}_1) = \lim_{k \to \infty} \sum \frac{1}{k^{3/2}}=0$}} 

\put(7,1.25){\makebox(0,0){$\displaystyle \Delta_3 (p_1^{*},p_1) =  |y_1 - y|$}} 

\put(1.5,1){\circle*{.1}}
\put(2,1){\circle*{.1}}
\put(3.5,1){\circle*{.1}}


 \put(1.3,0.7){\makebox(0,0){$p=p_0=(x,y)$}} 
\put(4.25,0.75){\makebox(0,0){$p_1^{*}$}} 
\put(3.75,3.25){\makebox(0,0){$p_1=(x_1,y_1)$}} 
\put(6.25,2.75){\makebox(0,0){$p_2^{*}$}} 
\put(6.75,4.25){\makebox(0,0){$p_2=q=(x',y')$}} 

\end{picture} 

\caption{Here $D_M(p,q)=\max \{ |x - x'|^{1/2}, |y-y'|^{1/3}\}$ }
\label{myfig}
\end{center}
\end{figure}
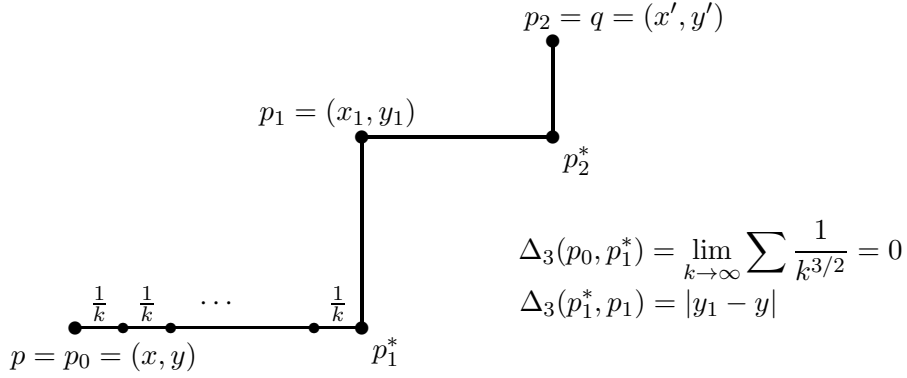

\begin{lemma}\label{trianglelemma}
$\triangle_{\alpha_{i}}(p,q)=0$ if and only if $p,q\in \R^n$ may be written in the form $p=(x,y)$ and $q=(x',y)$ where $y=(x_i, \cdots,x_r)$.
\end{lemma}
\begin{proof}
Suppose $p=(x,y)$ and $q=(x',y)$ where $x=(x_1,\cdots,x_{i-1})$ and $x'=(x'_1,\cdots,x'_{i-1})$. For $k>0$ pick $p_j=(x_j,y)$ such that $|x_{j-1}-x_j|=1/k$. Then 
$$[D_M(p_{j-1},p_j)]^{\alpha_{i}}=[D_i^+(x_{j-1},x_j)]^{\alpha_{i}}\leq 1/k^{\alpha_{i}/\alpha_l}$$
where $\alpha_l < \alpha_i$. Hence
$$\lim_{k\to \infty} \sum_{j=1}^{m} [D_M(p_{j-1},p_j)]^{\alpha_i}\leq\lim_{k\to \infty}  \sum_{j=1}^{m}\frac{1}{k^{\alpha_i/\alpha_l}}=0.$$
Conversely, suppose $p=(x,y)$ and $q=(x',y')$ where $y\neq y'$. (See figure \ref{myfig} above.) Given an arbitrary sequence $\{p_j =(x_j,y_j) \}$, 
define a new sequence by inserting a point $p_j^{*}=(x_{j},y_{j-1})$
between $p_{j-1}$ and $p_{j}$. Now $p_{j-1}$ and  $p_j^{*}$ have the same second coordinate and therefore $\triangle_{\alpha_i}(p_{j-1},p_j^{*})=0$. 
Also, $p_j^{*}$ and $p_{j}$ have the same first coordinate so
$$[D_M(p_{j}^{*},p_j)]^{\alpha_i}= [D_i^+(y_{j-1},y_j)]^{\alpha_i}.$$
We refine our sequence further by inserting points between  $p_{j-1}$ and $p_j^{*}$ as above to form a sequence $\{q_i\}$ such that
$$ \sum_{j=1}^{m} [D_M(p_{j-1},p_j)]^{\alpha_i}+\epsilon \geq \sum [D_M(q_{j-1},q_j)]^{\alpha_i} \geq \sum [D_i^+(y_{j-1},y_j)]^{\alpha_i}>0$$
demonstrating that $\triangle_{\alpha_i}(p,q) >0$.
\end{proof}

\noindent Now suppose $F$ is a $K$-$Bilip_{D_M}$ map. Then, given a sequence points $\{ q_i \}_{i=1}^{m}$ in $\R^n$,  where $q_0=F(p)$ and $q_m=F(q)$, define a sequence $\{p_i\}$ such that  $F(p_i)=q_i$. By definition
we know that for all $\beta>0$
$$ \frac{1}{K^\beta} \sum_{j=1}^{m}  [D_M(p_{i-1},p_i)]^\beta\leq \sum_{j=1}^{m}  [D_M(F(p_{i-1}),F(p_i))]^\beta\leq K^\beta \sum_{j=1}^{m}  [D_M(p_{i-1},p_i)]^\beta$$
and therefore by the definition of $\triangle_\beta$
$$\frac{1}{K^\beta} \triangle_\beta(p,q) \leq \triangle_\beta(F(p),F(q)) \leq K^\beta  \triangle_\beta(p,q).$$
This shows, by Lemma \ref{trianglelemma},  
that $F(x_1,\cdots, x_r)=(f_1(x_1,\cdots, x_r), \cdots, f_r( x_r ))$. 
Furthermore, the proof of Lemma \ref{trianglelemma} shows that $$\triangle_{\alpha_i}((x_1,\cdots,x_i,x_{i+1},\cdots,x_r),(x_1',\cdots,x_i',x_{i+1},\cdots,x_r))=|x_i-x_i'|$$ so that $f_i(x_i,\cdots,x_r)$ is $K^{\alpha_i}$ bilipschitz with respect to $x_i$:
$$\frac{1}{K^{\alpha_i}}|x_i-x_i'| \leq|f_i(x_i,x_{i+1},\cdots,x_r)-f_i(x_i',x_{i+1},\cdots,x_r)| \leq K^{\alpha_i}|x_i-x_i'|.$$
Also, to see that $f_i(x_i,\cdots ,x_r)$ is H\"{o}lder continuous in $x_l$
with exponent $\alpha_i/\alpha_l$ for $l>i$,
we fix $x_j$ where $j\neq l$ and let 
$$y=(x_{i+1},\cdots,x_l,\cdots,x_r),\ \ y'=(x_{i+1},\cdots,x_l',\cdots,x_r).$$
Then
\bea
D_M(F(x_1,\cdots,x_i,y),F(x_1,\cdots,x_i,y'))&=&\\
\max\{\cdots, |f_i(x_i,y)&-&f_i(x_i,y')|^{1/\alpha_i},\cdots \}\leq K|x_l-x_l'|^{1/\alpha_l}\eea
and so
 $$|f_i(x_i,y)-f_i(x_i,y')|\leq K^{\alpha_i} |x_l-x_l'|^{\alpha_i/\alpha_l}$$
 which demonstrates that $f_i$ is H\"{o}lder continuous.
\end{proof}

\noindent{As} a consequence of Proposition \ref{foliationlemma}, we get the following corollary:
%
%
\begin{cor} $Sim_{D_M}(\R^n)$ consists of maps which are the composition of a standard dilation $\delta_t$ along with a map of the form
$$ (x_1,x_2, \ldots, x_r) \mapsto (A_1 (x_1 +  B_1), \ldots, A_r(x_r + B_r))$$
where $A_i \in O(\R^{n_i})$ and $B_i \in \R^{n_i}$.
\end{cor}
\begin{proof}
Let $F$ be a $Sim_{D_M}$ map. Then for some $t\in \R$,  $D_M(F(p),F(q))=t D(p,q)$ for all $p,q \in \R^n$. 
If we write $p=(x_1, \ldots, x_r)$ then by Proposition \ref{foliationlemma} 
$$F(p)=(f_1(x_1, \ldots, x_r), \ldots, f_i(x_i, \ldots, x_r),\ldots, f_r(x_r))$$
where $f_i$ is a similarity  of $\R^{n_i}$ with similarity constant $t^{\alpha_i}$. 
Therefore $f_i$ has the form
$$f_i(x_i, \ldots , x_r)=  t^{\alpha_i} A_{(x_{i+1}, \ldots , x_r)} (x_i + {B}_{(x_{i+1}, \ldots , x_r)})$$
where $A_{(x_{i+1}, \ldots , x_r)} \in O(\R^{n_i})$ and ${B}_{(x_{i+1}, \ldots , x_r)} \in \R^{n_i}$. Now by Proposition \ref{rotconstprop} in Section \ref{rotation} we can conclude that $A_{(x_{i+1}, \ldots , x_r)}$ does not depend on $(x_{i+1}, \ldots ,x_r)$.  We will show that ${B}_{(x_{i+1}, \ldots , x_r)}$ is also fixed for all   $(x_{i+1}, \ldots ,x_r)$. First, consider $f_1$ and suppose that
${B}_{(x_{2}, \ldots , x_r)} \neq {B}_{(x'_{2}, \ldots , x'_r)} $. Then pick $x_1,x'_1$ such that  for all $j\neq1$  we have
$$|x_1-x'_1|^{1/\alpha_1} \geq |x_j-x_j'|^{1/\alpha_j}$$ and 
$$|f_1(x_1, \ldots, x_r)- f_1(x'_1, \ldots, x'_r)|^{1/\alpha_1} \geq |f_j(x_j, \ldots, x_r)- f_j(x'_j, \ldots, x'_r)|^{1/\alpha_j}.$$ 
This is possible because $f_j$ does not depend on $x_1$ for $j\neq1$.
 Then we would have the contradiction
  $$|x_1-x_1' + {B}_{(x_{2}, \ldots , x_r)} -  {B}_{(x'_{2}, \ldots , x'_r)}| = |x_1-x_1'|.$$ 
Now suppose for $j<i$ that each $f_j$ does not depend on $x_{l}$ for all $l>j$. Then we can repeat the above argument to show that ${B}_{(x_{i+1}, \ldots , x_r)}$ does not depend on $(x_{i+1}, \ldots ,x_r)$. 
\end{proof}
\subsection{Geometry of the solvable Lie group ${G_M}$.}\label{geomofGM}

We will briefly describe the construction and geometry of the solvable Lie group 
$$G_M = \R \ltimes_M \R^n$$
from Theorem \ref{polyrigid}.  Then we will describe the notions of \emph{height respecting} quasi-isometries of  $G_M$ and define a boundary for $G_M$.\\

\noindent{\bf Geometry.}
After squaring $M$ if necessary, we can ensure that $M$ lies on a one parameter subgroup $M^t$ in $GL(n,\R)$. Then the group $G_M$  has multiplication given by
$$(t,x) \cdot  (s,y)= (t+s, x+M^ty)$$
for all $(t,x),(s,y) \in \R\ltimes_M \R^n$.
We endow $G_M$ with the left invariant Riemannian metric given by the symmetric matrix:
$$\bm 1 & 0 \\ 0 & Q_M(t) \fm$$
where $Q_M(t)=(M^{-t})^TM^{-t}$.
For each $t$, this metric gives us a distance formula 
$$d_t(x,y)=|| M^{-t}(x-y)||.$$
In fact, it is possible to define $G_M$ not only when $\det M = 1$ but also for any matrix $M$ with $\det M>0$. When $M$ is a scalar matrix $\lambda I$ with $\lambda>1$, the group $G_M$ is isometric to hyperbolic space with curvature depending on $\lambda$. When $M$ has all eigenvalues greater than $1$ (or all eigenvalues less than $1$), the group $G_M$ is a negatively curved homogeneous space. When $M$ has eigenvalues both greater than, and less than $1$, as in the case of Theorem \ref{polyrigid}, $G_M$ admits two natural foliations by negatively curved homogeneous spaces, one arising from the eigenvalues greater than one, and the other arising from the eigenvalues less than one. For more details see Section 4 from \cite{FM3}.\\

\noindent{\bf Absolute Jordan form.}  Any lattice in $G_M$ is also a lattice in 
in the isometry group of $G_{\bar{M}}$ where $\bar{M}$  is the \emph{absolute Jordan form} of $M$ (see \cite{FM3} for details).
Therefore, we will replace $M$ with its absolute Jordan form and reorder the eigenvalues so that 
$$M=\bm M_l & \\ & M_u^{-1} \fm$$
where $M_l$ is an $n_l \times n_l$ matrix with diagonal entries
$e^{\alpha_i}$ with $\alpha_{i+1} > \alpha_i > 0$ and $M_u$ is an $n_u
\times n_u$ matrix with diagonal entries $e^{\beta_i}$
with $\beta_{i+1} > \beta_i > 0$.\\

\noindent{}Consider $G_{M}$ with the coordinates
$(t,x_1,\cdots x_r,z_1, \cdots, z_s)$. Now the Riemannian metric on
$G_{M}$ is given by
$$dt^2+e^{-2\alpha_1t}dx_1^2+\cdots e^{-2\alpha_rt}dx_r^2 + e^{2\beta_1t}dz_1^2+\cdots + e^{2\beta_st}dz_s^2.$$

\noindent{\bf Negatively curved homogeneous spaces.}  If  $M=M_l$ or $M=M^{-1}_u$ then $G_M$ is a negatively curved homogenous space.  Otherwise, for each fixed $z=(z_1,\ldots,z_s)$ we obtain a totally geodesic embedded negatively curved homogeneous space isometric to $G_{M_l}$. If we vary $z$ then we get a foliation of $G_M$ by spaces isometric to $G_{M_l}$.  Call this foliation $\F_l$. 
Similarly by fixing $x=(x_1,\ldots,x_r)$ and replacing $t$ by $-t$ we get another
foliation; this time by spaces isometric to $G_{M_u}$.  Call this
foliation $\F_u$. \\

 \noindent{}{\bf Height respecting.} Let $h: \R\ltimes_M \R^n \to \R$ be projection onto the first factor:
 $$(t,x_1,\ldots x_r,z_1, \ldots, z_s) \mapsto t.$$
  We will call this the \emph{height function} and call $t$ the \emph{height} of 
  the point $(t,x_1,\ldots x_r,z_1, \ldots, z_s)$. A quasi-isometry of $G_M=\R \ltimes_M \R^n$ is \emph{height respecting} if it permutes the level sets of $h$, up to bounded distance. By proposition 5.8 in \cite{FM3}, a height respecting quasi-isometry induces a map that is a bounded distance from a translation on the height factor.\\

\noindent{}{\bf Key Theorem.} The main ingredient in the proof of Theorem \ref{polyrigid} is Eskin-Fisher-Whyte's Theorem 2.2 in \cite{EFW} which states that if $\det{M}=1$ then 
all quasi-isometries of $G_M$ are height respecting. The proof of this theorem can be found in \cite{P}.\\

\noindent{}{\bf Vertical geodesics.} We say $\geod \in
G_M$ is a \emph{vertical} geodesic if 
 it is one of the form $\geod(t)=(-t,a_1,\cdots a_r,b_1, \cdots, b_s)$ or $\geod(t)=(t,a_1,\cdots a_r,b_1, \cdots, b_s)$. In the first case we say $\geod$ is \emph{downward oriented}; in the second case we say $\geod$ is \emph{upward oriented}. \\

\noindent{\bf Boundaries.}  If $M$ has all eigenvalues greater than one then $G_M$ is a negatively curved homogeneous space and so its (visual) boundary is simply $S^n$. However, since all of the maps we are interested in fix a common point, we can make the identification $\partial G_M \simeq \R^n$. Another way of describing this boundary is to identify $\partial G_M$ with the space of vertical geodesics in $G_M$.  Even when $G_M$ is not negatively curved, we still have a useful notion of a boundary for $G_M$. The
lower boundary $\partial_l G_M$ (upper boundary $\partial^u G_M$) can
be defined as equivalence classes of vertical geodesic $\geod$
which are downward oriented (upward oriented). Two downward (upward) oriented geodesics $\geod, \geod'$ are equivalent if $d_{G_M}(\geod(t),\geod'(t)) \to 0$ as $t \to \infty$. \\

\noindent{\bf Boundary maps induced by quasi-isometries.} 
In Section \ref{boundary} we will show that if $M$ is a diagonal matrix with all eigenvalues greater than one then any height respecting quasi-isometry of $G_M$ induces a $QSim_{D_M}$ map of $\partial G_M$.\\

\noindent{}For general $G_M$, we use Proposition 4.1 from \cite{FM3} which says that any height respecting quasi-isometry $\varphi$ of $G_M$ is a bounded distance from a quasi-isometry which preserves the two foliations $\F_l$  and $\F_u$.
So if $\varphi$ is a height respecting quasi-isometry that maps $L \in \F_l$ to within a bounded distance of $L'
\in \F_l$ then there is an induced map $\bar \varphi :\partial L \to
\partial L'$. Since $L$ is isometric to $G_{M_l}$ and since $\varphi$ is a
height-respecting quasi-isometry, $\bar \varphi$ is a $QSim_{D_{M_l}}$
map. Note that since $L$ is isometrically embedded and totally
geodesic in $G_M$ then $\partial L \subset \partial_l G_M$.  To see that the opposite inclusion also holds, note that for
any downward oriented vertical geodesic ray $\geod(t)=(x_o,z_o,-t)$
there is a downward oriented geodesic $\geod'(t)=(x_o,z_o',-t)$
contained in $L$ and at a bounded distance from $\geod$. Therefore
we can identify $\partial L$ with  $\partial_l G_M$ and so
a height respecting quasi-isometry induces a $QSim_{D_{M_l}}$ map of 
$\partial_l G_M \simeq \R^{n_l}$. Similarily, a height respecting quasi-isometry induces a $QSim_{D_{M_u}}$ map of $\partial_u G_M \simeq \R^{n_u}$.\\

%
%
\subsection{Relating $(\R^n, D_M)$ to $\partial G_M$.}\label{boundary}

In this section, let $M$ be a diagonal matrix with all eigenvalues greater than one. Recall that this condition on the eigenvalues ensures that $G_M$ is a negatively curved homogeneous space. We show how to relate height respecting quasi-isometries of $G_M$ with $QSim_{D_M}$ maps of $\R^n$.

%
%
\begin{lemma}\label{hrlemma} A height respecting quasi-isometry (resp. isometry) of $G_M$ induces
a $QSim_{D_M}$ map (resp. $Sim_{D_M}$ map) of $\partial G_M \simeq \R^n$. 
\end{lemma}
\begin{proof}
Given two points $p,q \in \R^n$, let $t$ be the height at which the two vertical geodesics emanating from $p$ and $q$ are at distance one apart. 
Now at height $t$ the distance between $p$ and $q$ is given by
$$d_t(p,q)=|| M^{-t}(p-q)||$$ where $M^{-t}$ is the $n\times n$ diagonal matrix
$$M^{-t}=\bm  e^{-t\alpha_1} & 0 & \cdots & 0\\
0 & e^{-t\alpha_2}&  \cdots & 0\\
\vdots& & \ddots  & \vdots\\
0& \cdots & 0 & e^{-t\alpha_r}
\fm$$
and so $$d_t(p,q)=\max\{ e^{-t\alpha_1}|x_1-y_1|, \cdots , e^{-t\alpha_r}|x_r-y_r|\}.$$
Now, if this maximum occurs in the $i$th coordinate then we have
$$e^{-t\alpha_i}|x_i-y_i|=1$$
and so
$$e^t=|x_i-y_i|^{1/\alpha_i}.$$
Setting ${D_M}(p,q)=e^t$, we get that 
$${D_M}(p,q)=\max \{|x_1-y_1|^{1/\alpha_1}, \cdots ,|x_r-y_r|^{1/\alpha_r}\}.$$
A height respecting isometry maps level sets of height $t$ to level sets of height $t + a$. Therefore, if $\phi$ is the boundary map induced by this isometry then 
$${D_M}(\phi(p),\phi(q))=e^{t+a}=e^a {D_M}(p,q).$$
A height respecting quasi-isometry, after composing with a height respecting isometry,  induces a map which is at a bounded distance from the identity on the $t$ coordinate. Let $\e$ be this bound. Let $F$ be the boundary map induced by a height respecting quasi-isometry and $t'$ is the height at which the vertical geodesics emanating from $F(p)$ and $F(q)$ are distance one apart. Then 
$$t-\e \leq t' \leq t+\e.$$
Hence, $$ e^{-\e} e^{t} \leq e^{t'} \leq e^{\e} e^{t}.$$
Since $d_{t'}(F(p),F(q))=1$, for some $j$ we have $e^{-t'\alpha_i}|x_j-y_j|=1$
and so
$$e^{t'}=|x_j-y_j|^{1/\alpha_i}.$$
Therefore,
$$1/K'\ {D_M}(p,q)\leq {D_M}(F(p),F(q)) \leq K'\ {D_M}(p,q)$$
where $K'=e^{\e}$. Thus $F$ is a $QSim_{D_M}$ map as required.
\end{proof}

\noindent{\bf Boundary versus $\mathbf{G_M}$.} Lemma \ref{hrlemma} allows us to view a group that acts on $G_M$ by height respecting quasi-isometries as acting by $QSim_{D_M}$ maps on $\R^n\simeq \partial G_M$. In fact, the proof of Lemma \ref{hrlemma} shows that a $(K,C)$-quasi-isometry that induces the map $t \mapsto t+a$ on the height factor induces a $(e^a,K')$-$QSim_{D_M}$ map of $\R^n$ where $K'$ depends only on $K$ and $C$.\\ 

\noindent{\bf Defining a quasi-action.} The converse to Lemma \ref{hrlemma} is also true. If $G$ is a $QSim_{D_M}$ map of $\R^n$ then for any $a\in \R$
$$\phi(x_1, \cdots, x_r,t)=(G(x_1, \cdots, x_r),\ t + a)$$
is a quasi-isometry of $G_M$.
The quasi-isometry constants of $\phi$ depend on the $QSim_{D_M}$ constants of $G$ as well as on $a$. In particular, if $G$ is a $(N,K)$-$QSim_{D_M}$ map then for $a=\log{N}$ the map $\phi$ is a $(K,1)$ quasi-isometry of $G_M$.\\

\noindent{}{\bf Space of distinct pairs.}
However, if our goal is to define a quasi-action on $G_M$ by a uniform group of $QSim_{D_M}$ maps then this formula is insufficient. 
The problem with defining a quasi-action using this method is that you have to know the $QSim_{D_M}$ constants in order to define $\phi$ and these constants are not unique. 
 To facilitate going back and forth between groups acting on the space $G_{M}$ and $\partial G_{M}$, we define the following space: 
%
 $$\mathcal{P}= \{(p,q) \mid p,q \in \partial G_M \simeq \R^n, p\neq q\}.$$
This is the space of distinct pairs of points of $\R^n$. We define a map 
$$\rho: \mathcal{P} \to\ G_M$$ as follows:
for any $(p,q) \in \partial G_M$ consider the vertical geodesics in $G_M$ emanating from $p$ and $q$. At some height $t_o$, these two geodesics are distance one apart. Define 
$\rho$ by setting 
$\rho(p,q)=(p,t_o)$.
This map is onto and has compact kernel. 
Furthermore, if $\phi$ is a height respecting quasi-isometry of $G_M$ and $F$ is the induced $QSim_{D_M}$ boundary map, then there exists a constant $C$, depending only on the quasi-isometry constants of $\phi$, such that for all $(p,q)\in \mathcal{P}$
$$d_{G_M} ( \phi(z), \rho( F(p), F(q)))< C$$
where $\rho(p,q)=z \in G_M$, and $d_{G_M}$ denotes distance in $G_M$.
In this way, any uniform group of $QSim_{{D_M}}$ maps that acts cocompactly on the space of distinct pairs of points of $\R^n$ can also be treated as a group which quasi-acts coboundedly on $G_{M}$ by height-respecting quasi-isometries.

\subsection{The quasi-isometry group $QI(G_M)$}\label{QIgroup}
In this section we return to the case where $M$ is a diagonalizable $n\times n$ matrix with $\det{M}=1$. We will describe the quasi-isometry group $QI(G_M)$. 

\begin{prop} Up to finite index, $$QI(G_M) \simeq QSim_{D_{M_l}}(\R^{n_l}) \times QSim_{D_{M_u}}(\R^{n_u})$$
where $M_l$ and $M_u$ are as defined in Section \ref{geomofGM}.
\end{prop}
\begin{proof}
From Section \ref{geomofGM} we know that any quasi-isometry 
 $\phi: G_M \to G_M$ induces a $QSim_{D_{M_l}}$ map $G_l$ of $\R^{n_l}$
 and a $QSim_{D_{M_u}}$ map $G_u$ of $\R^{n_u}$. Furthermore, if two quasi-isometries are a bounded distance apart then they induce the same maps $G_l$ and $G_u$.
If two quasi-isometries are not at a bounded distance then they induce different boundary maps. Therefore we have an injection 
$$QI(G_M) \hookrightarrow QSim_{D_{M_l}}(\R^{n_l}) \times QSim_{D_{M_u}}(\R^{n_u}).$$
To see that this map is actually onto we can use the same ideas as in Section \ref{boundary}. Namely, given any $G_l\in QSim_{D_{M_l}}(\R^{n_l})$ and $G_u\in QSim_{D_{M_u}}(\R^{n_u})$ we can construct a quasi-isometry $\phi$ of $G_M$ by setting 
$$\phi(x,z,t)=(G_l(x),G_u(z),t).$$
Finally, note that if $M_l=M_u$ then $G_M$ has an extra isometry:
$(x,z,t)\mapsto (z,x,-t)$ so that in this case
$$QI(G_M) \simeq QSim_{D_{M_l}}(\R^{n_l}) \times QSim_{D_{M_u}}(\R^{n_u})\rtimes \Z_2.$$
\end{proof}
\section{Proving Theorem \ref{mytukia2} }\label{tukiasection}

We will write $D_M$ instead of $D$ as in the statement of the Theorem in order to emphasize the relationship between the metric $D$ and the solvable Lie group $G_M$ (see Section \ref{notation} for details). Note that in this section $M$ has all eigenvalues greater than $1$. (The results of this section will be applied to $M=M_u$ and $M=M_l$.)
We prove Theorem \ref{mytukia2} by induction on the number of distinct $\alpha_i$ occuring in $D_M$. 
The base case, when there is only one distinct $\alpha_i$, is discussed in Section \ref{basecase}.  We set up the induction step in Section \ref{inductionstep}. In Section \ref{onedim} we start proving the case when $n_1=1$
and in Section \ref{moredim} when $n_1>1$ . In Sections \ref{multconst} and \ref{rotation} we combine both $n_1=1$ and $n_1>1$ to finish the proof of Theorem \ref{mytukia2}.

\subsection{Base Case.}\label{basecase}

In this section, we prove the base case of Theorem \ref{mytukia2}. In other words, we will prove Theorem \ref{mytukia2} in the case where $D_M$ is of the form

$$D_M(x,y)= |x-y|^{1/\alpha}.$$ 
From Section \ref{boundary}, we know that $Sim_{D_M}$ and $QSim_{D_M}$ maps of $\R^n$ correspond to height-respecting isometries and quasi-isometries of the solvable Lie group
$$G_M = \R \ltimes_M \R^n$$
where $M$ is the scalar matrix $e^\alpha I$. In this case, $G_M$ with the coordinates $(t,x)$ has a Riemannian metric give by 
$$ dt^2 + e^{-2\alpha t}dx^2$$
which makes $G_M$ isometric to hyperbolic space $\hnn$ with curvature depending on $\alpha$.
Notice that in this special case a $K$-$Bilip_{D_M}$ map of $\R^n$ is simply a $K^\alpha$-bilipschitz map of $\R^n$ with respect to the standard metric on $\R^n$. Likewise, a $K$-$QSim_{D_M}$ map of $\R^n$ is a $K^\alpha$-quasisimilarity of $\R^n$ with respect to the standard metric. Therefore, we are reduced to studying uniform groups of quasisimilarities of $\R^n$ with respect to the standard metric.

\begin{prop} Let $\G$ be a uniform group of quasisimilarities of $\R^n$. If $n>1$, assume further that $\G$ acts cocompactly on the space of distinct pairs of points of $\R^n$.  Then there exists a quasisimilarity $F:\R^n \to \R^n$ such that 
$$F\G F^{-1} \subset Sim(\R^n).$$
\end{prop}
\begin{proof}
The case when $n=1$ is Theorem 3.2  in \cite{FM2}.
When $n>1$, we can apply Tukia's theorem as follows: any uniform group of quasisimilarities of $\R^n$ can be treated as a uniform group of quasiconformal maps of $S^n$. Also, the action of $G$ on the space of distinct triples of $S^n$ is cocompact since $\G$ has a universal fixed point and acts cocompactly on the space of distinct pairs of points of $\R^n$. Applying Tukia's theorem, we get a quasiconformal map $F$ such that
$$F\G F^{-1} \subset Conf(S^n).$$
However, we need something stronger; we need $F$ to be a quasisimilarity. 
Luckily, Tukia constructs $F$ explicitly and in the case when $\G$ is a group of quasisimilarities, $F$ is also a quasisimilarity.
Therefore $$F\G F^{-1} \subset Conf(S^n)\cap QSim(\R^n) \simeq Sim(\R^n).$$
See Section \ref{moredim} for more details on the construction of $F$.\end{proof}

\subsection{Setting up the induction step.}\label{inductionstep}
 
In this section, we set up the induction step needed to prove Theorem \ref{mytukia2}. We consider the space $(\R^n, D_M)$ where
$$D_M(x,y)=\max \{|x_1-y_1|^{1/\alpha_1}, \cdots ,|x_r-y_r|^{1/\alpha_{r}}\}$$
where $0<\alpha_i < \alpha_{i+1}$ and $x=(x_1,\cdots,x_r), y=(y_1, \cdots, y_r)$ with $x_i,y_i \in \R^{n_i}$.
From Section \ref{boundary}, we know that $QSim_{D_M}$ maps of $\R^n$
are the maps we need to consider when studying height-respecting isometries and quasi-isometries of the solvable Lie group
$$G_M = \R \ltimes_M \R^n$$
where $M$ is the diagonal matrix with diagonal entries $e^{\alpha_i}$. Note that each $e^{\alpha_i}$ occurs $n_i$ times in $M$. In this case, $G_M$ with the coordinates $(t,x_1,\cdots,x_r)$ has a Riemannian metric give by 
$$dt^2+e^{-2\alpha_1 t }dx_1^2 + \cdots +e^{-2\alpha_r t} dx_r^2$$
which makes $G_M$ isometric to a negatively curved homogeneous space.
From Proposition \ref{foliationlemma} in Section \ref{notation}, we know that a $QSim_{D_M}$ map $F$ has the form
$$ F(x_1,x_2, \cdots, x_r)=(f_1(x_1,\cdots, x_r), \cdots, f_r( x_r ))$$
where $f_i$ is bilipschitz in the variable $x_i$. \\

\noindent{Next}, set $n'=n-n_1$ and consider the metric $D_{M'}$ on $\R^{n'}$ given by
$$D_{M'}(x',y')=\max \{|x_2-y_2|^{1/\alpha_2}, \cdots ,|x_r-y_r|^{1/\alpha_{r}}\}$$
where $x'=(x_2, \cdots, x_r)$ and $y'=(y_2, \cdots, y_r)$. With this notation we can write
$$D_M(x,y)=\max\{ |x_1-y_1|^{1/\alpha_1}, D_{M'}(x',y')\}.$$
\\
\noindent{\bf Induced representations.}  We have an induced representation
$$\phi: QSim_{D_{M}} (\R^n)\twoheadrightarrow QSim_{D_{M'}}(\R^{n'}),$$ 
given by 
 $$(f_1(x_1,\cdots, x_r), \cdots, f_r( x_r )) \mapsto (f_2(x_2,\cdots, x_r), \cdots, f_r( x_r )).$$
Given a uniform subgroup $\G \subset QSim_{D_M} ( \R^n),$  we can consider the image $$\phi(\G)\subset QSim_{D_{M'}}(\R^{n'}).$$
By induction, since $D_{M'}$ has fewer distinct $\alpha_i$'s, there exists a $QSim_{D_{M'}}$ map $F'$ such that each element of $F' \phi(\G) {F'}^{-1}$ is an $ASim_{D_{M'}}$ map.
We can pick some element of $\phi^{-1}(F')$ and conjugate $\G$ by this element. This gives us an action on $\R^n$ by maps of the form 
$$G(x_1,y)=(g_y(x_1), g(y))$$
where $y=(x_2, \cdots, x_r)$, and
 $g(y)$ is a $ASim_{D_{M'}}$ map. 
We can now focus on altering $g_y(x_1)$ to make $G$ into a $ASim_{D_M}$ map.
We treat the one dimensional case ($n_1=1$) in Section \ref{onedim}  and higher dimensional cases  ($n_1\geq 2$) in Section \ref{moredim}.
In each of these cases,
we first show that we can conjugate $\G$ to a group where $g_y(x_1)$ is a similarity of $\R^{n_1}$ for each fixed $y$; that is, 
$$g_y(x_1)= \lambda_y A_y (x_1 + B(y))$$
where $A_y \in O(n_1)$, $B(y) \in \R^{n_1}$ and $\lambda_y \in \R_+$.
In Section \ref{multconst}, we show that we can conjugate $\G$ again so that not only does $\lambda_y$ not depend on $y$ but it also matches up with the similarity constant of $g(y)$. We also show that $A_y$ must be independent of $y$. This is done in Section \ref{rotation} and concludes the proof of Theorem \ref{mytukia2}.

%
%

\subsection{One-dimensional case.}\label{dim1}\label{onedim}

At this point, we have a group $\G$ acting uniformly by $QSim_{D_M}$ maps on $\R^n \simeq \R \oplus \R^{n'}$
where each group element $G \in \G$ is of the form

$$G(x,y)=(g_y(x), g(y))$$
where $x \in \R, y \in \R^{n'}$, and
 $g(y)$ is an $ASim_{D_{M'}}$ map.
 We also know from Section \ref{notation} that for $y$ fixed $g_y(x)$ is a bilipschitz map of $\R$. In fact, we have a foliation $\F$ of $\R^n$ by copies of $\R$ where each leaf corresponds to a fixed $y\in \R^{n'}$.

The main idea of this section is to find a $\G$-invariant measure along each leaf of this foliation. To this end, we will use Rademacher's Theorem that any Lipschitz function, in particular $g_y(x)$, is almost everywhere differentiable.\\

\noindent{\bf Bounded derivative.} Since $g_y$ is bilipschitz, $g_y'(x)$ is bounded, but it is not uniformly bounded over all elements $G \in \G$.  However, since $g(y)$ is an $ASim_{D_M}$ map, it has a well-defined similarity constant $t_g$. This allows us to write 
$$ G(x,y)=\delta_{t_g}(\bar{g}_y(x),\bar{g}(y))$$
where $\bar{g}_y'(x)$ is now uniformly bounded over all $G \in \G$.
Set $$\overline{G}=\delta_{t_{g}}^{-1}G.$$

\noindent{}{\bf Finding an invariant measure.}  Let $U\subset \R\oplus \R^{n'}$ be a $\G$ invariant set of full measure such that $g_y'(x)\neq 0$ for all $G\in \G$ and for all $(x,y) \in U$.
For each $G \in \G$, define a map $\mu_G$ on $U$  by
$$\mu_G(x,y) = \bar{g}_y'(x).$$
We will construct a $\G$ invariant measure from the $\mu_G$s. Let
$$M_{(x,y)}=\{ \mu_G(x,y) \mid G \in \G\}=\{ \bar{g}'_y(x) \mid G \in \G\},$$
then if $H\in \G$ 
\bea
 M_{H(x,y)}&=&\{ \bar{g}_{h(y)}'(h_y(x)) \mid G \in \G\}\\
&=&\left\{\frac{\overline{(g\circ h)}_y'(x)}{\bar{h}_y'(x)} \mid G \in \G\right\}\\
&=&\frac{1}{\bar{h}_y'(x)} M_{(x,y)}. 
\eea
Since $M_{(x,y)}$ is bounded, we can define
$$\mu(x,y)= \sup M_{(x,y)}$$
which satisfies
$$\mu(H(x,y))=\frac{1}{\bar{h}_y'(x)}\mu(x,y).$$

\noindent{\bf Finding a conjugating map.}  First, we define a family of metrics $\nu_y$ on $\R$ so that for all $H \in \G$ the map $h_y$ is a similarity with respect to these metrics. For $y\in \R^{n'}$ and $x_1,x_2\in \R$ define

$$\nu_y(x_1,x_2)= \int_{x_1}^{x_2} \mu(t,y) dt.$$
We check that 
\bea
\nu_{h(y)}(h_y(x_1),h_y(x_2))&=& \int_{h_y(x_1)}^{h_y(x_2)} \mu(H(t,y))dh_y(t)\\
					      &=& \int_{x_1}^{x_2} \mu(H(t,y))\delta_{t_h} h_y'(t) dt\\
				 	      &=& \delta_{t_h} \int_{x_1}^{x_2} \mu(t,y) dt\\
					      &=& \delta_{t_h} \nu_y(x_1,x_2).
\eea
If  we set 
$$F(x,y)=(\nu_y(x,0),y)$$
then $F$ is a $QSim_{{D_M}}$ map and for each $H \in F \G F^{-1}$ the map $h_y$ is a similarity of $\R$.

%
%

\subsection{Higher dimensional case.}\label{dimbigger}\label{moredim}
In this section, $\G$ is a group that acts uniformly by $QSim_{D_M}$ maps on $\R^n \simeq \R^{n_1}\oplus \R^{n'}$ where $n_1>1$, and acts cocompactly on the space of distinct pairs of points of $\R^n$. Also, each group element $G \in \G$ is of the form 
$$ G(x,y)=(g_y(x),g(y))$$
 where $x \in \R^{n_1} , y \in \R^{n'}$ and $g(y)$ is an $ASim_{D_{M'}}$ map.  From Section \ref{bilipDM} we also know that 
$g_y(x)$ is a quasisimilarity of $\R^{n_1}$, hence differentiable almost everywhere with derivative bounded in terms of the quasisimilarity constants. Again, by fixing $y$, we have a foliation of $\R^n$ by copies of $\R^{n_1}$.

In this case, we follow Tukia's proof, found in \cite{T}, on conjugation groups of quasiconformal maps of $S^n$ into the group of conformal maps.
First, we define a notion of a \emph{${D_M}$-foliated conformal structure} on $\R^n$: a conformal structure defined on the sub-bundle of the tangent bundle consisting of  subspaces corresponding to the directions parallel to the leaves of the $\R^{n_1}$ foliation.

Next, we define an action of $\G$ on the space of all ${D_M}$-foliated conformal structures on $\R^n$. In section \ref{invconfstr}, we show there exists a $\G$ invariant  ${D_M}$-foliated conformal structure on $\R^n$ (Theorem \ref{confstruct}). Using this structure, in section \ref{conjgroup}, we prove that we can find a conjugating map $F$ such that 
each element $G \in F\G F^{-1}$ has the form
$$G(x,y)=(g_y(x),g(y))$$
where $g_y(x)$ is a similarity of $\R^{n_1}$ and $g(y)$ is an $ASim_{D_{M'}}$ map.\\

%
%
\subsubsection{Notation and Definitions.}
\noindent{}{\bf ${D_M}$-foliated conformal structure.}
A \emph{${D_M}$-foliated conformal structure} $\mu$ is an assignment
for almost every $(x,y) \in \R^n$,
	$$\mu(x,y) \in SL(n_1,\R)/SO(n_1,\R),$$
 such that $\mu(x,y)$ is a measurable map.
We can identify $SL(n_1,\R)/SO(n_1,\R)$ with the space of all symmetric, positive definite $n_1 \times n_1$  real matrices with determinant one via the map $M \mapsto M^TM$.  If we make this identification then the group $GL(n_1,\R)$ acts on 
$SL(n_1,\R)/SO(n_1,\R)$ by
$$X[A]=|det X|^{-2/n} X^T A X.$$
%
%
 We can define a distance on $SL(n_1,\R)/SO(n_1,\R)$ by
 $$k(I,A)=\max\{\log{\lambda_{max}}, \log{1/\lambda_{min}} \}$$
where $\lambda_{max}$ is the largest eigenvalue of $A$, and $\lambda_{min}$ is the smallest eigenvalue.
Alternatively,
$$d(I,A)=\sqrt{(\log{\lambda_1})^2 + \cdots + (\log{\lambda_{n_1}})^2}.$$
We can extend either of these metrics to all of $SL(n_1,\R)/SO(n_1,\R)$ by requiring them to be invariant under 
the action of $GL(n_1,\R)$. We define the dilation of a matrix $A$ to be 
$$K(A)=\exp{k(I,A)}.$$

%
%
\noindent{}{\bf Conformal.} We will say that a $QSim_{D_M}$ map $F(x,y)=(f_y(x),f(y))$ is conformal in $\mu$ if
\bea
	\mu(x,y) = f_y'(x)\left[ \mu(F(x,y))\right].
\eea
%
%
We write
	$$\mu_F(x,y)
	=f_y'(x)[I]$$
$$ \mu_{FG}(x,y)= g_y'(x)\left[\mu_F(G(x,y))\right]$$
$$F_{*}\mu(F(x,y))=f_y'(x)^{-1}[\mu(x,y)].$$
%
%
If $\mu$ is a conformal stucture and $A$ the symmetric positive definite matrix which represents $\mu(x,y)$ then define
$$K(\mu(x,y))=K(A)$$
$$K(F(x,y))=K(\mu_F(x,y)).$$
%
%
Recall that if $F(x,y)=(f_y(x),f(y))$ is a $K$-$QSim_{D_M}$ map then 
$f_y(x)$ is a $K$-quasisimilarity of $\R^{n_1}$. This implies that $K(F(x,y))\leq K$.\\

\noindent{}We also have that  if $G$ is conformal in $\mu$ then 
\begin{equation}
K(G(x,y)) \leq K(\mu(x,y)) K(\mu(G(x,y))).
\end{equation}
%
%

\subsubsection{Lemmas on $QSim_{D_M}$ maps.}
In this section we collect some lemmas on how measurable sets behave under $QSim_{D_M}$ maps. In the statements of the lemmas, $m$ will denote spherical measure. However, since spherical measure and Lebesgue measure are comparable on bounded sets, we will prove the lemmas using Lebesgue measure. Note that the measure induced by $D_M$ on $\R^n$ is the usual Lebesgue measure.
From \cite{T} we have the following:\\
%
%

\noindent{}{\bf Lemma B1 in \cite{T}}\emph{ Let $\F$ be a compact family of $K$-quasiconformal embeddings $U \to \bar{\R}^n$ with $U \subset \bar{\R}^n$ open. Then there are positive $a,a'$ depending on $K$  and $b,b'$ depending on $K$ and $\F$ such that
$$b'm(E)^{a'} \leq m(f(E))\leq bm(E)^a$$ 
for all measurable $E$.
}\\

\noindent{}For $QSim_{D_M}$ maps, the following lemma replaces the previous one:

\begin{lemma}\label{MeasureEst} Let $\F$ be a family of $(N,K)$-$QSim_{D_M}$ maps where $K,N$ is fixed and each $F\in \F$ has the form
$$F(x,y)=(f_y(x),f(y))$$  
where $f_y(x)$ is bilipschitz and $f(y)$ is a $ASim_{D_{M'}}$ map with similarity constant $N$.
Then there exists $b$ and $b'$ such that for each $F \in \F$ and each bounded measurable $E\subset \R^n$,
$$b'm(E) \leq m(F(E))\leq bm(E).$$ 
\end{lemma}
\begin{proof}
For any bilipschitz map $f$ of $\R^n$  and any measurable set $E \subset \R^n$, we have 
$$\frac{1}{K^n}m(E) \leq m(f(E)) \leq K^nm(E).$$
We treat the usual Lebesgue measure $m$ on $\R^n$ as a product measure $m_1 \times m_2$, where $m_1$ is Lebesgue measure on $\R^{n_1}$ and $m_2$ is Lebesgue measure on $\R^{n'}$.
Using Fubini's Theorem, we write
$$m(E) = \int_{\R^{n'}} m_1(E\vert_y) dm_2, \ \textrm{ and } \ 
m(F(E)) = \int_{\R^{n'}} m_1(F(E)\vert_y) dm_2.$$
Since $f_y$ is bilipschitz, we have the following estimate on $m_1(F(E)\vert_y)$: 
$$\frac{1}{K'}m_1(E\vert_y) \leq m_1(F(E)\vert_{f(y)}) \leq K'm_1(E\vert_y).$$
If $f(y)$ were a similarity then to compute $m(F(E))$ we would need only to calculate
$$ \int_{\R^{n'}} m_1(F(E)\vert_{f(y)}) |f'(y) |dm_2.$$
Now a similarity has constant derivative $|f'(y)|=N'$ where $N'$ depends only on the similarity constant and on the dimension of $\R^{n'}$. 
Combining these two facts we would get
$$\frac{N'}{K'}\int_{\R^{n'}}m_1(E\vert_y)dm_2 \leq \int_{\R^{n'}}m_1(F(E)\vert_{f(y)})|f'(y)|dm_2 \leq N'K' \int_{\R^{n'}}m_1(E\vert_y)dm_2$$
or rather
$$\frac{N'}{K'}m(E)\leq m(F(E))\leq N'K' m(E).$$
Note that in the general case, where $f(y)$ is an $ASim_{D_{M'}}$ map, we can prove this lemma by induction. Namely, if we have the following estimate for the measure of $f(E\vert_{x})$: 
$$R m_2(E\vert_{x_1})\leq m_2(f(E\vert_{x_1}))\leq R' m_2(E\vert_{x_1})$$
then
$$m(F(E)) \leq  \int_{\R^{n'}}R' \ m_1(F(E)\vert_{f(y)})dm_2 \leq R'N'K' m(E)$$
and
$$\frac{RN'}{K'} m(E)\leq  \int_{\R^{n'}}R' \ m_1(F(E)\vert_{f(y)})dm_2 \leq m(F(E)).$$
\end{proof}

%
%
\begin{lemma}\label{converges}\label{showsim} Let $H_i:U \to \R^n$ be a family of $K$-$QSim_{D_M}$ maps  such that $H_i \to H$ for some $H \in QSim_{D_M}(\R^n)$.  Suppose that for all $\e >0$
$$m(\{(x,y) \in U : K(H_i(x,y))\geq 1 + \e \}) \to 0.$$
Then 
$$ H(x,y)=(h_y(x),h(y))$$
where $h_y(x)$ is a similarity.
\end{lemma}
\begin{proof}
Set $$A_i^{\e}= \{(x,y) \in U : K(H_i(x,y))\geq 1 + \e \}.$$
Then for all $\beta>0$ there exists an $I$ such that if $i> I$ then $m(A_i^{\e}) < \beta.$ 
We want to show that for a fixed $y$  the map $h_{y}$ is a similarity. 
We already know that $h_{y}$ is a $K$-quasisimilarity of $\R^{n_1}$ and so is $K$-quasiconformal. Using the Lemma B2 from \cite{T}, we will show that $h_y$ is conformal and so, must be a similarity. \\

\noindent{\bf Lemma B2} \emph{Let $f_i: U \to \R^n$ be a family of $K$-quasiconformal embeddings such that  $f_i \to f$ and for all $\e >0$ 
$$m(\{ x\in U : K(f_i(x)) \geq 1 + \e \}) \to 0.$$
Then, $f$ is conformal.
}\\

\noindent Now ${h_i}_y$ form a family of $K$-quasiconformal maps,  and ${h_i}_y \to h_y$. We know that $m(A_i^\e) \to 0$ but we don't know whether $$m_1(A_i^{\e} \vert_y) \to 0.$$
Let $i$ be large enough so that $m(A_i^\e) < C_k/k$
where $C_k$ is the volume of the ball of radius $1/k$.\\

\noindent{}{\bf Claim:} There exists  a $y_k$  such that 
$$|y-y_{k}|<1/k \textrm{ and } m_1( A_i^{\e} \vert_{y_{k}}) < {1}/{k}.$$
If this were not the case, then we would get the the following contradiction:
$$m(A_i^\e) > \int_{|y-z|<1/k} m_1(A_i^\e \vert_z) dm_2(z) > C_k/k.$$
So now $h_{y_k} \to h_y$ and $$m(\{ x : K(h_{y_k}(x)) \geq 1 + \e \}) < 1/k. $$
By Lemma B2, $h_y$ is conformal and hence a similarity.
\end{proof}

%
%
\subsubsection{Invariant  ${D_M}$-foliated conformal structure.}\label{invconfstr}
The proof of the following theorem follows the proof found in \cite{T}.
\begin{thm} \label{confstruct} If $\G$ is a separable group of $QSim_{D_M}$ maps then $\R^n$ has a $\G$-invariant  ${D_M}$-foliated conformal structure.
\end{thm}
 \begin{proof}
Assume first that $\G$ is countable. Then there is a set of full measure  $U\subset \R^{n}$ such that for every $G \in \G$ the map $g_y$ is differentiable with non-vanishing and bounded Jacobian. Define the set
$$ M_{(x,y)}= \{ \mu_F(x,y) \mid F \in \G\}. $$
Then
\bea
    g_y'(x)\left[ M_{G(x,y)} \right]&=& \{ g_y'(x) \left[\mu_F(G(x,y))\right] \mid F \in \G\}\\
    	&=& \{\mu_{FG}(x,y) \mid F \in \G\}\\
	&=& M_{(x,y)}
\eea
Recall from \cite{T} that $S=SL(n_1,\R)/SO(n_1,\R)$ is a non positively curved space. So for each bounded subset $X\subset S$, there is a unique disk with center $P_X$ of smallest radius containing $X$.  Therefore, we can define a continuous map $$ X \mapsto P_X. $$ 
Set
	$$\mu(x,y)=P_{M_{(x,y)}}.$$
Then
 $$\mu(G(x,y)) = g_y'(x)\left[ \mu(x,y)\right]$$
 for all $G \in \G$
and so $ \mu(x,y)$  is $\G$ invariant. 
To see that $\mu$ is measurable, consider the following: First label the elements of $\G=\{ G_0,G_1, \cdots \}$ and define
 $$ M^j_{(x,y)}=\{ \mu_{G_i}:i\leq j\}$$
 and 
$$ \mu_j(x,y)=P_{M^j_{(x,y)}}.$$
Now since $\mu_G(x,y)$ is measurable, 
$\mu_j(x,y)$ is measurable for all $j$.  But $\mu_j(x,y) \to \mu(x,y)$ so that $\mu$ is also measurable.

In general, let $\G_o$ be a countable dense subset of $\G$ and let $\mu$ be a $\G_o$ invariant $D_M$-foliated conformal structure. We will show that $\mu$ is $\G$ invariant as well. To do this, we will follow Theorem D from \cite{T}. We will state Theorem D in the language of $QSim_{D_M}$ maps but we will not include any proof since its proof is identical to the proof in \cite{T}.\\

\noindent{\bf Theorem D} Let $F_i:U \to V$ be a sequence of $K$-$QSim_{D_M}$ maps such that $F_i \to F$ for some ($QSim_{
D_M}$) map $F:U \to V$. Suppose that for all $\e>0$ 
$$m(\{(x,y) \in U : K_{\mu,\nu}(F_i(x,y))> 1+\e\} ) \to 0$$
as $i \to \infty$. Then $F$ is  $(\mu,\nu)$ conformal.\\
\\

\noindent{}Here 
$$K_{\mu,\nu}(F(x,y))=\exp{ k(\mu(x,y),f'_y(x)\left[\nu(F(x,y))\right])}$$ and
we say a $QSim_{D_M}$ map $F(x,y)=(f_y(x),f(y))$ is $(\mu,\nu)$ conformal if 
$$\mu(x,y) = f_y'(x)\left[ \nu(F(x,y))\right].$$
To show that $\mu$ is $\G$ invariant consider the following: Given $F \in \G$, let $F_i \in \G_o$ be such that $F_i \to F$. Since $F_i\in  \G_o$ we have that $K_{\mu,\mu}(F_i(x,y))
=1$ and in particular
$$m(\{(x,y) \in U : K_{\mu,\mu}(F_i(x,y))> 1+\e\} ) = 0.$$
By Theorem D, $F$ is also conformal in $\mu$ 
and so $\mu$ must be $\G$ invariant.
\end{proof}

\subsubsection{Conjugating the group.}\label{conjgroup}
In order to state the theorem we prove in this section, we need two definitions:\\

\noindent{}{\bf Radial point.}  Recall that we have a map from the space of distinct pairs of points of $\partial G_M$ to $G_M$ 
$$\rho: P \to G_M.$$
We call a point
 $p \in \R^{n}\simeq \partial G_M$ a \emph{radial point} of $\G$ if there exists a sequence of elements $G_{i} \in \G$, a point $z=\rho(q_1,q_2)\in G_M$ and a geodesic $L \in G_M$ with endpoint $p$ such that 
 $z_i=\rho(G_{i}(q_1),G_{i}(q_2)) \to p \textrm{ and } d_{G_M}(z_i,L)\leq C 
\textrm{ for all } i.$ By an abuse of notation we will write $z_i=G_i(z)=\rho(G_{i}(q_1),G_{i}(q_2))$.\\
 
%
%
\noindent{}{\bf Approximate continuity.} Let $U\subset \R^n$ and $(X,d)$ be a metric space. An open map $f: \R^n \to X$  is approximately continuous at $x\in U$  if for all $\e > 0$
$$\lim_{r \to 0} \frac{m(\{y\in B(x,r) \cap U\mid d(f(x),f(y)) \leq \e\})}{m(B(x,r))} =1.$$
In \cite{T}, Tukia notes that if $X$ is a separable metric space and if $f$ is measurable with respect to the Borel sets of $X$ and Lebesgue measurable sets of $U$, then $f$ is a.e. (with respect to $m$) approximately continuous.
We apply this definition to the map that defines our $\G$ invariant $D$-foliated conformal structure
$$\mu: \R^n \to SL(n_1,\R)/SO(n_1,\R).$$
Recall that the metric on $SL(n_1,\R)/SO(n_1,\R)$ is given by $d(A,B)=\log(K(A^{-1}B)).$  So if $\mu$ is approximately continuous at $p$ then 
$$\lim_{r\to 0} \frac{m(\{q \in B(p,r) \cap U\mid K(\mu(q)) > 1+\e\})}{m(B(p,r))} = 0.$$
Here $B(p,r)$ is the ball of radius $r$ around the point $p$.
%
%
\begin{thm}\label{TukiaRadAbs}
Let $\G$ be a uniform group of $QSim_{D_M}$ maps that all have the form
$$G(x,y)=(g_y(x),g(y))$$
where $g(y)$ is an $ASim_{D_{M'}}$ map. Let $\mu$ be a $\G$-invariant ${D_M}$-foliated conformal structure. If  $p$ is a radial point for $\G$ and $\mu$ is approximately continuous at $p$, then there exists a $QSim_{D_M}$ map $F$ such that every $H\in F\G F^{-1}$
has the form
$$H(x,y)=(h_y(x),h(y))$$
where $h(y)$ is again a $ASim_{D_{M'}}$ map and $h_y(x)$ is a similarity of $\R^{n_1}$. 
\end{thm}

\begin{proof} First, we will construct the conjugating map $F$ as a limit of group elements composed with dilations. Without loss of generality, let $p=(0,\cdots,0)\in \R^n$ be our radial point. Choose a map $\al \in QSim_{D_M}(\R^n)$ that fixes $p$ and with the property that $\al_{*}\mu(p)=Id$. For instance, we can choose $\al(x,y)=(Tx,y)$ where $T$ is a linear map that sends the matrix
$\mu(p)$ to the identity matrix. 

Let $G_i\in\G$ be the maps which make $p$ a radial point. We write $G_i(x,y)=(g_{iy}(x),g_i(y))$.
We can pick real numbers $t_i$ such that  $\delta_{t_i}\circ G_i$ is actually $K$-$Bilip_{D_M}$. 
In particular, we can chose the $t_i$ so that $t_i^{-1}$ is the similarity constant of $g_i(y)$. Note that $t_i \to \infty$ but if we fix $z \in P$ then the diameter of $\{(\delta_{t_i}\al G_i)(z)\}$ is bounded by a constant that depends only on $C$. 
Define $F_i$ by
$$ F_i(x,y)=\delta_{t_i}\al G_i(x,y)=\al \delta_{t_i} G_i(x,y).$$
%
%
Since each $F_i$ is $K$-$Bilip_{{D_M}}$ for a fixed $K$, and since $\{F_i(z)\}$ is bounded, we have that $F_i \to F$ where $F$ is also a $K$-$Bilip_{{D_M}}$ map.
Next, we need to show that for all $G\in \G$, the map
$FGF^{-1}(x,y)=(g_y(x),g(y))$ is such that $g_y(x)$ is a similarity of $\R^{n_1}$.
Let 
$$H_i=F_iGF_i^{-1}=\delta_{t_i}\al G_iGG_i^{-1}(\delta_{t_i}\al)^{-1}$$
and consider the conformal structure
$$\mu_i(x,y)={F_i}_{*}\mu(x,y).$$
Note that $H_i$ is conformal in $\mu_i$ and that 
\bea
\mu_i(x,y)&=&{F_i}_{*}\mu(x,y)\\
          &=&{\delta_{t_i}}_{*}\al_{*}{G_i}_{*}\mu(x,y)\\
          &=&{\delta_{t_i}}_{*}\al_{*}\mu(x,y)\\ 
&=&e^{-\alpha_1 t_i}I[\al_{*}\mu(\delta_{t_i}^{-1}(x,y))]\\
&=&\al_{*}\mu(\delta_{t_i}^{-1}(x,y)).\\
\eea
%
%
Now since $\mu$ is approximately continuous
at $p$, so is $\al_{*}\mu$.
Therefore 
$$\lim_{r \to 0}\frac{m(\{(x,y)\in B(p,t_ir) \cap U: K(\mu_i(x,y)) > 1+ \e\})}{m(B(p,r))} \to 0.$$
Let $A_i \subset \R^n$ be sets such that $K(\mu_i(x,y)) > 1+ \e$,
then $m(A_i) \to 0$.
Now let $B_i$ be the sets on which $K(H_i(x,y))> 1+ \epsilon$. Then, by Lemma  \ref{MeasureEst} we have $m(B_i) \to 0$ as well.
Set  ${C_i}={A_i} \cup {B_i}$. Then, on $\R^{n}\setminus {C_i}$, we have both 
$$ K(\mu_i(x,y)) \leq 1 + \e \quad \textrm{and} \quad K(\mu_i(H_i(x,y))) \leq 1+\e.$$
So,
$$K(H_i(x,y)) \leq K(\mu_i(x,y)) K(\mu_i(H_i(x,y)))\leq (1+\e)^2.$$
Now we can apply Lemma \ref{showsim} to show that $H=\lim_{i\to \infty} F_iGF_i^{-1}$ has the form $(h_y(x),h(y))$ where $h_y$ is a similarity of $\R^{n_1}$.

Finally, since $ASim_{D_{M'}}$ maps form a group, we know that if $G(x,y)=(g(x,y),g(y))\in \G$ then, after conjugating by $F$, the map $g(y)$ is still an $ASim_{D_{M'}}$ map. So, we only need to show that the conjugating map $F(x,y)=(f(x,y),f(y))$ is such that $f(y)$ is an  $ASim_{D_{M'}}$ map.
To see this, we note that
$$ F = \lim_{i \to \infty} F_i = \lim_{i \to \infty} \delta_{t_i}aG_i$$ where $G_i\in \G$, the map $\delta_{t_i}$ is a standard dilation with respect to $D_M$, and $a$ is a linear map that is the identity on $y$. Since $G_i \in  \G$, we have that $G_i(x,y)=(g_i(x,y),g_i(y))$ where $g_i(y)$ is an $ASim_{D_{M'}}$ map. Therefore, the limiting map $f(y)$ is a $ASim_{D_{M'}}$ map. This completes the proof of Theorem \ref{TukiaRadAbs}.
\end{proof}

\noindent{}To complete the proof of the multidimensional case we observe that if $\G$ acts cocompactly on the space of distinct pairs of points of $\R^n$ then every point is a radial point (see \cite{T} or \cite{C} for details).

\subsection{Uniform multiplicative constant.}\label{mult}\label{multconst}

At this point, we have a group $\G$  of  $QSim_{D_M}$ maps of the form
\bea G(x,y)&=&(g_y(x),g(y)) \\
	            &=&(\lambda_{g,y}A_{y}(x +B_{y}),g(y))
\eea	            
where $A_{y} \in O(n_1)$,  $\lambda_{g,y} \in \R^+$, 
and $g(y)$ is an $ASim_{D_{M'}}$ map with similarity constant $t_g \in \R^+$.\\

\noindent The goal of this section is to show that by conjugating with an appropriate map we can eliminate the dependence of  $\lambda_{g,y}$ on $y$
and determine the multiplicative constant to be $t_g^{\alpha_1}$.\\

\noindent By composing  $G(x,y)$ with the inverse of the standard dilation
$\delta_{t_g}$
we have 
$$\delta_{t_g}^{-1}\circ G (x,y)=(\eta_{g,y}A_{y}(x +B_{y}),\bar{g}(y))$$
where $$\eta_{g,y}=\frac{\lambda_{g,y}}{t_g^{\alpha_1}}.$$
Set
$$M_y=\{\eta_{g,y} \mid G \in \G\}.$$
Since $\lambda_{(g\circ f),y}= \lambda_{f,y} \lambda_{g,f(y)},$
\bea
M_{f(y)}&=&\{\eta_{g,f(y)}\mid G\in \G\}\\
	      &=&  \left\{\frac{\lambda_{(g\circ f),y}/\lambda_{f,y}}{t_g^{\alpha_1}} \mid G\in \G\right\}\\
	      &=& \left\{\frac{\lambda_{(g\circ f),y}}{t^{\alpha_1}_g t_f^{\alpha_1}}\frac{t^{\alpha_1}_f}{\lambda_{f,y}}\mid G\in \G\right\}\\
		&=& \frac{1}{\eta_{f,y}} M_y.
\eea
Now, since $\eta_{g,y}$ is universally bounded,  we can define
$$\mu(y)=\sup M_y$$
which has the property
$$\mu(f(y))= \frac{1}{\eta_{f,y}}\mu(y).$$
A simple calculation gives that conjugating $\G$ by $F(x,y)=(\mu(y)x,y)$ gives an action by elements of the form 
$$(x,y) \mapsto \delta_t ( A_y(x+B_y),\bar{g}(y))$$
where $\bar{g}$ is an $ASim_{D_{M'}}$ with similarity constant one.

\subsection{Uniform rotation constant.}\label{rotation}		

At this point, we have a group $\G$ where each element has the form
\bea G(x,y)=(g_y(x),g(y)) 
	            =(t^{\alpha_1}  A_{y}(x +B_{y}),g(y))
\eea	            
where $A_{y} \in O(n_1)$, and $g(y)$ is an $ASim_{D_{M'}}$ map.
The goal of this section is to show that $A_{y}$ does not depend on $y$. 
We do this by showing that if $A_{y} \neq A_{y'}$ then $G$ is not a $K$-$QSim_{D_M}$ map.

\begin{prop}\label{rotconstprop} Let $G(x,y)=(t^{\alpha_1} A_{y}(x +B_{y}), g(y) )$  
be a $K$-$QSim_{D_M}$ map as above. Then $A_{y}=A_{y'}$ for all $y,y' \in \R^{n_2}$.
\end{prop}
\begin{proof}
Suppose that for some $y,y'$ we have that $A_{y} \neq A_{y'}$.  Pick $z\in \R^{n_1}$ such that $A_{y} z \neq A_{y'} z$. In fact, we can pick $z$ so that for any $N$
$$|A_{y} z -A_{y'} z| > N.$$
Next, pick $x,x'$ such that  $x+B_{y}=z$ and $x'+B_{y'}=z$.
Note that 
$$\vert x-x'\vert = | B_{y} -B_{y'} | \leq K^{\alpha_1} D_{M'}(y,y')^{\alpha_1}$$ 
so that 
$$ {D_M}((x,y),(x',y'))= \max \{ |x-x'|^{1/\alpha_1},D_{M'}(y,y')\} \leq K D_{M'}(y,y')$$
but
$${D_M}(G(x,y),G(x',y'))=\max \{ t  |A_{y} z -A_{y'} z|^{1/\alpha_1},D_{M'}(g(y),g(y'))\}.$$
We can now make ${D_M}(G(x,y),G(x',y'))$ arbitrarily large by picking $z$ such that 
$|A_{y} z -A_{y'} z| $ is arbitrarily large. 
In particular, there exists a $z$ so that 
$$ {D_M}(G(x,y),G(x',y')) > K^2 D_{M'}(y,y') >K {D_M}((x,y),(x',y')).$$
This violates the assumption that $G$ was a $K$-$QSim_{D_M}$ map.
\end{proof}

\section{Application: Quasi-isometric Rigidity.}\label{endgame}
In this section, we show how Theorem \ref{mytukia2} and the work of Eskin-Fisher-Whyte \cite{EFW} can be used analyze the structure of groups that are quasi-isometric to lattices in the solvable Lie groups
$$G_M=\R \ltimes_M \R^{n}$$ where $M$ is a diagonalizable matrix with $\det{M}=1$ and no eigenvalues on the unit circle.
Recall from Section \ref{geomofGM} that we can replace $M$ with its absolute Jordan form so that 
$$M=\bm M_l & \\ & M_u^{-1} \fm$$
where $M_l$ $M_u$ are diagonal matrices with all eigenvalues greater than one.

Now, any lattice in a solvable Lie group must be a cocompact lattice. Therefore, if $\Gamma$ is quasi-isometric to a lattice in $G_M$ then $\Gamma$ is also quasi-isometric to $G_M$ itself. Let $\varphi: \Gamma \to G_M$ be a quasi-isometry. Then, $\Gamma$ quasi-acts properly and coboundedly on $G_M$ by quasi-isometries of the form 
$$\varphi L_G \bar{\varphi}$$
 where $\bar{\rho}$ is a coarse inverse of $\rho$ and $L_G$ denotes left multiplication in $\Gamma$. \\

\noindent{\bf Key Theorem  \cite{EFW}.} Let $G_M$ be as defined above. Then every quasi-isometry of $G_M$ is height respecting.\\

\noindent{}Recall from section \ref{geomofGM}  that this key theorem allows us to construct a representation 
$$\Gamma \mapsto \G \subset QSim_{D_{M_l}}(\R^{n_l})\times QSim_{D_{M_u}}(\R^{n_u}).$$

In Section \ref{nonameyet} we will use Theorem \ref{mytukia2} to conjugate the image of this representation $\G$ into a subgroup of $ASim_{D_{M_l}}(\R^{n_l})\times ASim_{D_{M_u}}(\R^{n_u})$. In Section \ref{showinpoly} we use this structure to show that if $\Gamma$ is a finitely generated group then $\Gamma$ must be polycylic. 
In Section \ref{RbyRn}, we finsh the proof of Theorem \ref{polyrigid} by showing that $\Gamma$ is virtually a lattice in 
$\R\ltimes_{M'} \R^n$ where $M'$ has the same absolute Jordan form as $M^\alpha$ for some $\alpha \in \R$.

\subsection{Action by almost isometries.}\label{nonameyet}
For our purposes right now, we can 
now suppose $\G'$ is a group quasi-acting by
quasi-isometries on $G_M$. (For our purposes right now we can assume $\G'$ is our $\Gamma$ from above, however later on we will need the arguments from this section to apply to more general groups, so we state the results in more generality.)
 \noindent{Then} we have a representation of
$\G'$ onto $\G \subset QSim_{D_{M_l}}(\R^{n_l})\times
QSim_{D_{M_u}}(\R^{n_u})$. That is, for $G \in \G'$ we have
$$  G \mapsto (G_l,G_u) \in QSim_{D_{M_l}}(\R^{n_l})\times QSim_{D_{M_u}}(\R^{n_u})$$
where $G_l$
acts on the lower boundary and $G_u$
acts on the upper boundary. Since two quasi-isometries that
are at a bounded distance from each other induce the same boundary
maps,  the kernel of $\G' \to \G$ may be non-trivial.  Note also that if $\G'$
is a uniform group of quasi-isometries, then there is an $\epsilon$, fixed over all group elements, such that 
each element
induces a map on the height factor that is within $\epsilon$ of a
translation. (For any height-respecting quasi-isometry, the bound $\epsilon$ only depends on the quasi-isometry constants.)
Therefore, by the proof of Lemma \ref{hrlemma} the maps $G_u$ and $G_l$ are $(N,K)$-$QSim_{D_{M_l}}$ and
$(1/N,K)$-$QSim_{D_{M_u}}$ maps respectively, where $K=e^\epsilon$ is
fixed and $N$ is determined by the amount of translation on the height
factor. Finally, if the quasi-action of $\G'$ on $G_M$ is cobounded,
then the action of $\G$ on the space of distinct pairs of points of
$\partial_l G_M$, and on the space of distinct pairs of points of
$\partial_u G_M$, is cocompact. (See the end of section \ref{boundary} for details.)

%
%
\begin{prop}\label{twoconj} If $\G \subset
  QSim_{D_{M_l}}(\R^{n_l})\times QSim_{D_{M_u}}(\R^{n_u})$ is a
  separable group that is induced by a uniform group of
  quasi-isometries which quasi-acts coboundedly on $G_M$, then there
  exists an $$F \in QSim_{D_{M_l}}(\R^{n_l})\times
  QSim_{D_{M_u}}(\R^{n_u})$$ such that $F \G F^{-1}$ consists of
  elements $(G_l,G_u)$ where $G_l$ is an $ASim_{D_{M_l}}$ map and
  $G_u$ is an $ASim_{D_{M_u}}$ map.
\end{prop}
\begin{proof}
  We will apply Theorem \ref{mytukia2} twice.  Consider the projection
$$ \pi_l: QSim_{D_{M_l}}(\R^{n_l})\times QSim_{D_{M_u}}(\R^{n_u}) \to QSim_{D_{M_l}}(\R^{n_l})$$
$$G=(G_l,G_u) \mapsto G_l$$
and let $\G_l=\pi_l(\G)$ be the image of this projection.
Then $\G_l$ is a uniform subgroup of $QSim_{D_{M_l}}(\R^{n_l})$ that
acts cocompactly on the space of distinct pairs of points of
$\R^{n_l}$. By Theorem \ref{mytukia2} there is a $QSim_{D_{M_l}}$ map
$F_l$ such that $F_l \G_l F_l^{-1}$ consists of $ASim_{D_{M_l}}$ maps.
Similarily, we can define $\G_u$ and apply Theorem \ref{mytukia2}
again to get a $QSim_{D_{M_u}}$ map $F_u$ such that $F_u \G_l
F_u^{-1}$ consists of $ASim_{D_{M_u}}$ maps. Then, setting
$F=(F_l,F_u)$, we have $F\G F^{-1}$ as desired. \end{proof}
%
%
\begin{prop}\label{stretchinverse} Let $F\G F^{-1}$ be as above and
  let $G\in F\G F^{-1}$. That is, $G=(G_l,G_u)$ where $G_l$ is an
  $ASim_{D_{M_l}}$ map and $G_u$ is a $ASim_{D_{M_u}}$ map. Let $t_l$
  and $t_u$ be the similarity constants of $G_l$ and $G_u$
  respectively. Then $ t_l= 1/t_u$.
\end{prop}
\begin{proof}
  The map $G$ has the form
$$G=(G_l,G_u)=(F_u G'_u F_u^{-1},F_l G'_l F_l^{-1})=(\delta_{t_l}\circ H_l, \delta_{t_u} \circ H_u)$$
where $G'_u$ is a $(K,N)$-$QSim_{D_{M_l}}$ map and $G'_l$ is a
$(K,1/N)$-$QSim_{D_{M_u}}$ map. Also, $F_u$ and $F_l$ are
$K'$-$Bilip_{D_{M_l}}$ and $K'$-$Bilip_{D_{M_u}}$ maps respectively,
so that if $t_u$ is the similarity constant of $F_u G_u F_u^{-1}$ and
$t_l$ the similarity constant of $F_l G_l F_l^{-1}$ then
 $$t_u \in \left[\frac{N}{KK'^2} , K'^2KN\right] \textrm{ and } t_l \in \left[\frac{1}{KK'^2N}, \frac{KK'^2}{N}\right].$$ 
 In particular, the ratios of the interval endpoints is $K''=K^2K'^4$ and is
 fixed over all group elements. So for the map $G$, we know that
 $$1/K'' \leq t_u/t_l \leq K''.$$
 Furthermore, we know that $G^k$ has upper boundary similarity
 constant $t_u^k$ and lower boundary similarity constant $t_l^k$.
 Suppose that $t_u$ and $t_l$ were not inverses of each other. Then
 for large enough $k$ we could make $(t_u/t_l)^k$ arbitrarily large
 (or small) violating that $1/K'' \leq t_u^k/t_l^k \leq K''$.
\end{proof}

\noindent{\bf Action on $\R^n$.} Explicitly, at this point, $\G$ acts
on $\R^n=\R^{n_l}\times \R^{n_u}$ by maps of the form
$G(x,w)=(G_l(x),G_u(w))$ where
$$G_l(x)=(t_l^{\alpha_1}A^l_1(x_1+B^l_1(x_2, \cdots, x_{r_l})), \cdots, t_l^{\alpha_r}A^l_r(x_1+B^l_r))$$
$$G_u(x)=(t_l^{-\beta_1}A^u_1(w_1+B^u_1(w_2, \cdots, w_{r_u})), \cdots, {t_l^{-\beta_{r_u}}}A^u_r(w_1+B^u_r))$$
where $A_i^l\in O({n_l}_i)$ and $A_i^u \in  O({n_u}_i)$.\\

\noindent{}{\bf Action on $G_M$.} If $G_l$ and $G_u$ were
actually $Sim_{D_{M_l}}$ and $Sim_{D_{M_u}}$ maps we would be able to
define an action of $\G$ on $G_M$ by isometries by setting
$$G(x,z,t)=(G_l(x),G_u(z),t+\log{t_l}).$$
In our case, when $G_l$ and $G_u$ are $ASim_{D_{M_l}}$ and
$ASim_{D_{M_u}}$ maps, we call this an action by \emph{almost
  isometries}. We call the group of all such maps $AIsom(G_M)$.  We
will treat
$$AIsom(G_M) \subset ASim_{D_{M_l}}(\R^{n_l})\times ASim_{D_{M_u}}(\R^{n_u})$$ since $AIsom(G_M)$ embeds into this group.\\

\noindent{\bf Height homomorphism.} On $AIsom(G_M)$ there is a well
defined \emph{height homomorphism} $h:AIsom(G_M) \to \R$ given by
$$h(G)=\log{t_l}=\log{1/t_u}.$$ 

%
%
\subsection{Showing Polycyclic.}\label{showinpoly}

Recall that a group is \emph{polycyclic} if it has a has a descending
normal series where all quotients of consecutive terms are finitely
generated abelian.  In this section, we will show that if $\Gamma$ is
a finitely generated group quasi-isometric to $G_M$ then $\Gamma$ must
be virtually polycyclic. We will state and prove the theorems of this section in more generality than is needed for groups quasi-isometric to $G_M$ since this work is also used by Peng in \cite{P} to show rigidity of lattices in more general solvable Lie groups of the form $\R^k \ltimes \R^n$.\\

\noindent{}For $1 \leq i \leq r$ let $M_i$ be $n_i \times n_i$ diagonal matrices with entries greater than one.  Define
 $$\mathcal{A}=ASim_{D_{M_1}}(\R^{n_1})\times \cdots \times ASim_{D_{M_s}}(\R^{n_s}).$$
Let $S=\{ \vec{v_1},\ldots \vec{v_r}\}$ be a spanning set for $\R^k$ such that no two vectors are positive multiples of each other.
Define the \emph{uniform subgroup} of $\mathcal{A}$ with respect $S$ to be 
$$ \mathcal{U}_{S}=\{ (G_1, \ldots , G_r) \in \mathcal{A} \mid \textrm{ for some } \vec{v} \in \R^k,\ \log{t_i} = \left<\vec{v_i}, \vec{v}\right> \textrm{ for } i=1,\ldots,r   \} $$
where $t_i$ denotes the similarity stretch factor of $G_i$. Since $S$ spans $\R^k$ the vector $v$ is uniquely determined for each $G\in \mathcal{U}_S$ and so we can
define a \emph{stretch homomorphism} $$\sigma:U_S \to \R^k$$ by $\sigma(G)=v$.\\

\noindent{}In the case when $k=1$ and $S=\{1,-1\}$ the uniform subgroup reduces to our previous definition of $AIsom(G_M)$:
$$U_S=\{(G_l,G_u) \mid \log{t_l}= v,  \log{t_u}= -v \textrm{ for some } v\in \R \},$$
and the stretch homomorphism is the height homorphism from before.\\

\noindent{}Now each $G_i \in ASim_{D_{M_i}}(\R^{n_i})$ has the form 

$$\delta_{t_i} \circ (A_1(x_1+B_1(x_2,\cdots,x_{r})),\cdots,A_{r}(x_{r} +B_{r}))$$
where $(A_1, \ldots, A_r) \in O(\R^{n_i})$ and $\delta_{t_i}$ is a standard dilation with respect to $D_{M_i}$. For each $i$, we can define a homomorphism 
$$\psi_i: ASim_{D_{M_i}}(\R^{n_i}) \to O(\R^{n_i})$$ 
by $\psi_i(G_i)=(A_1, \ldots, A_r)$ and we can combine the $\psi_i$ to define the \emph{rotation homomorphism} on $\mathcal{A}$ by
$$ \psi(G)=(\psi_1(G_1), \ldots, \psi_s(G_s)).$$ 

\begin{lemma}\label{amenlemma} Suppose $\Gamma \subset \mathcal{A}$ is a finitely generated group quasi-isometric to a polycyclic group. Then $\psi(\Gamma)$ is abelian. 
\end{lemma}
\begin{proof} This follows from the fact that the only amenable subgroups of $O(n)$ are abelian. Since amenability is a quasi-isometry invariant, and since polycyclic groups are amenable, $\Gamma$ is also amenable and so $\psi(\Gamma)$ is amenable and hence abelian \cite{Gre}. 
\end{proof}

\noindent{}{\bf Almost translations.} Combining these two homomorphisms
we get a map
$$\sigma \times\psi: \mathcal{U}_S  \to \R^k \times O(n)$$
whose kernel consists of $G=(G_1, \ldots, G_s)$ where each 
$G_i$ is now an almost translation with respect to $D_{M_i}$
$$G_i(x_1, \ldots, x_r)=(x_1+B_1(x_2,\cdots,x_{r}),\cdots, x_{r} +B_{r}).$$

\noindent{}Define $\widetilde{M}$ to be the diagonal matrix obtained by combining
the diagonal entries of all of the matrices $M_i$ and reordering them from
smallest to largest. We can now think of $G\in ker(\sigma \times \psi)$ as acting on $\R^n$ by a 
$K$-$Bilip_{D_{\widetilde{M}}}$ almost translation.\\

\noindent{}In this section we will prove the following theorem:

\begin{thm}\label{polythm} Suppose $\G \subset \mathcal{U}_S$ is a finitely generated group that is quasi-isometric to a polycyclic group. Let 
$$\mathcal{N}= ker(\sigma\times \psi) \cap \Gamma.$$
Suppose further that $\mathcal{N}$ quasi-acts properly on $\R^n$ by $K$-$Bilip_{D_{\widetilde{M}}}$ almost translations. Then $\G$ is also (virtually) polycyclic.
\end{thm}
We can then apply Theorem \ref{polythm} to our situation where $\Gamma$ is quasi-isometric to a lattice in $G_M$. 

\begin{cor} Suppose $\Gamma$ is quasi-isometric to a lattice in $G_M$ then 
$\Gamma$ is polycyclic. 
\end{cor}
\begin{proof}
We already have a finite kernel representation of $\Gamma$ as a subgroup  $\G \subset \mathcal{U}_S$ where $k=1$ and $S=\{1,-1\}$. Now we need to show that $\mathcal{N}$ quasi-acts properly on $\R^n$. We already know that $\G$ quasi-acts properly on $G_M$ so that any subgroup of $\G$ must also act properly on $G_M$ and on any subset of $G_M$ it stabilizes. Since $\mathcal{N}$ stabilizes height level sets, and level sets are isometric to $\R^n$ then $\mathcal{N}$ must quasi-act properly on $\R^n$.
\end{proof}

\noindent{}The key ingredient in the proof of Theorem \ref{polythm} is the following proposition.

\begin{prop}\label{qaprop} Suppose a group $\mathcal{N}$ quasi-acts properly on $R^n$ by $K$-$Bilip_{D_{\widetilde{M}}}$ almost translations. Then $\mathcal{N}$ is finitely generated nilpotent.
\end{prop}

\noindent{}To prove Theorem \ref{polythm} we only need to note that if $\G$ is as in the statement of the theorem then by Lemma \ref{amenlemma}  $\G/\N$ is finitely generated abelian. By Proposition \ref{qaprop} we also have that $\N$ is finitely generated nilpotent. Therefore $\G$ is (virtually) polycyclic. \\

\noindent{\emph{Proof of Proposition \ref{qaprop}.}} We will show that $\N$ is a finitely generated nilpotent group by studying
its quasi-action on $\R^n$.  We start by proving
two key lemmas. The first one, Lemma \ref{Kevinslemma}, is due to
Kevin Whyte. The second one, Lemma \ref{Irineslemma}, was made
possible by an observation by Irine Peng.
%
%
\begin{lemma}\label{Kevinslemma}
  If a group $\N$ quasi-acts properly on $\R^n$ then any chain of finitely
  generated subgroups of $\N$ where each subgroup in the chain has infinite
  index in the next has length at most $n$.
\end{lemma}
\begin{proof}
  Let $\N$ be a finitely generated group quasi-acting properly on a
  metric space $(X,d)$.  Choose a basepoint $x_0$ of $X$.  Let
  $$f_\N(k) = |\{ \g \mid d(\g x_0,x_0) \leq k\}|$$
  {\bf Claim:} If $\Hcal$ is an infinite index subgroup of $\N$ then,
  for some $K$ and $C$,
  $$ K f_\N ( K k +C ) +C \geq n f_\Hcal(k).$$
  In other words, up to linear changes, the ratio grows linearly.  In
  particular, if $f_\Hcal$ has a polynomial lower bound of degree $n$
  then $f_\G$ has a lower bound of degree $n+1$.  If $X$ is $\R^n$
  then $f_\G$ is bounded above by a polynomial of degree $n$, and so
  the lemma follows.  To prove the claim, consider the set of $\g\in \N$
  which move $x_0$ at most $k$:
  $$S=\{ \g \mid d(\g x_0,x_0) \leq k\}.$$
  Divide this collection into $\Hcal$ cosets.  Let $\g_1, ...., \g_j$ be
  a collection of coset representatives.  Now for any $\m$ which moves
  $x_0$ at most $k$, and any $\g_i$, we have \bea
  d( \n_i \m x_0, x_0) &\leq& d(\n_i \m x_0, \n_i x_0) + d(\n_i x_0, x_0)\\
  &\leq& K d(\m x_0,x_0) + C + k \leq (K+1) k + C \eea so that
  $$f_\N ( (K+1) k + C)\geq j f_\Hcal(k).$$
  What remains is to see that $j$, the number of cosets of $\Hcal$ in
  $\N$ with representatives moving $x_0$ at most $k$, grows linearly
  with $k$.
  Pick any word metric on $\N$.
  Since each generator moves $x_0$ at most some $R$, the ball $\B_{k/R}$
  is contained in $S$, the set of elements moving $x_0$ at most $k$.
  Thus it suffices to see that the number of $\Hcal$ cosets
  represented in the ball of radius $r$
  in $\N$ grows linearly in $r$.\\

  \noindent{\bf Claim: }For every $r$,  there is a coset that intersects $\B_r$ but does not intersect $\B_{r-1}$.\\  

  \noindent{}Suppose not. Let $\Vcal$ be a coset of $\Hcal$ in $\N$ and let $\Vcal = \g\Hcal$ where $\g$ has minimal norm.
  Since $\Hcal$ has infinite index in $\N$, we can choose $\Vcal$ so that $|\g| > r +1$. We write $\g$ as $\g = \g_1\g_2$ where $\g_2$ is the first $r$ letters (from the right) in a minimal length word for $\g$  and $\g_1$ is the rest of the word. Then $|\g_1| + |\g_2| = |\g|$ adn $|g_2|=r$ . The coset $\g_2 \Hcal$ is also represented by a coset $\g_3 \Hcal$ where $d(\g_3,e) < r$ by assumption. Since $\g_2 \Hcal$  is in $\g_3 \Hcal$, we have that 
  $\g \Hcal = \g_1 \g_2 H = \g_1 \g_3 H$. But $|\g_1 \g_3| < |\g_1 \g_2 | = |\g|$. This contradicts the minimality of the norm of $\g$.\\

\noindent{}This shows that if $X$ has polynomial growth of degree $n$ then any
  chain of finitely generated subgroups each infinite index in the
  next can have length at most $n$.
\end{proof}

\noindent{}We will first show that $\N$, the kernel of $h\times \phi$, is finitely generated
polycyclic.  Once we show that $\N$ is finitely generated we will be
able to conclude that it is virtually nilpotent since it quasi-acts on
$\R^n$.
\begin{lemma}\label{Irineslemma}
  If $\g \in \N$, then $B_{i,\g}(y)$ is bounded as a function of $y$ and for any
  $(x_{i+1},\cdots, x_r)$ and $(x_{i+1}',\cdots, x_r')$
$$|B_{i,\g}(x_{i+1},\cdots, x_r)-B_{i,\g}(x_{i+1}',\cdots, x_r')| \leq \e_{i,\g}$$
where
$$ \e_{i,\g}=\max_{j>i}\{ 2K^{\alpha_i}(B^{max}_{j})^{\alpha_i/\alpha_j}\}$$
and $$B^{max}_{j,\g}= \sup_{y \in \R^{n_j} }\{B_{j,\g}(y)\}.$$

\end{lemma}\begin{proof}
We will work by induction. The case $i=r$ is clear, since the $B_{r,\g}$ are constants. We now assume the statements for $i+1$ and prove it for $i$.
  Let $x=(x_1, \cdots, x_i) $ and $y_i=(x_{i+1},\cdots x_r)$. For $\g
  \in N$ we have
$$\g(x,y_i)=(\cdots , x_i + B_{i,\g}(y_i), \cdots)$$
and
$$ \g^n(x,y_i)=(\cdots, x_i+ B_{i,\g^n}(y_i), \cdots).$$
By an abuse of notation we will also let
$$\g y_i= (x_{i+1} +B_{i+1,\g}(y_{i+1}), \cdots, x_r +  B_{r,\g}).$$
Consider $(x,y_i)$ and $(x,\g y_i)$.  Then \bea
D_M( (x,y) ,(x, \g y)) &=& \max_{j>i}\{ |B_{j,\g}(y_j)|^{1/\alpha_j}\}\\
&\leq& \max_{j>i}\{ (B^{max}_{j,\g})^{1/\alpha_j}\} \eea Note that
$B_{i,\g^n}(y_i)=B_{i,\g}(y_i)+B_{i,\g}(\g y_i) + \cdots +
B_{i,\g}(\g^{n-1}y_i)$ so that
$$D_M(\g^n(x,y_i),\g^n(x,\g y_i))=\max\{\cdots, |B_{i,\g}(y_i)-B_{i,\g}(\g^n y_i)|^{1/\alpha_i},\cdots \}$$
But we also have that
$$D_M(\g^n(x,y),\g^n(x,\g y)) \leq K D_M( (x,y) ,(x, \g y))$$
So that
$$ |B_{i,\g}(y_i)-B_{i,\g}(\g^n y_i)| \leq K^{\alpha_i}\max_{j>i}\{ (B^{max}_{{j,\g}})^{\alpha_i/\alpha_j}\}$$
Now in general we know that
$$|B_{i,\g}(y)-B_{i,\g}(y')|\leq K^{\alpha_i}D_{M_i}(y,y')^{\alpha_i}$$
and that
\bea |B_{i,\g}(y)+B_{i,\g}(\g y)+ \cdots +B_{i,\g}(\g^{n-1}y)\quad \quad \quad&&\\
- B_{i,\g}(y') - B_{i,\g}(\g y') \cdots -B_{i,\g}(\g^{n-1} y')|&\leq&
K^{\alpha_i}D_{M_i}(y,y')^{\alpha_i} \eea But we also know that for
all $k$
$$|B_{i,\g}(\g^k y)-B_{i,\g}(y)|<K^{\alpha_i}\max_{j>i}\{ (B^{max}_{j,\g})^{\alpha_i/\alpha_j}\}$$
and $$|B_{i,\g}(\g^ky')-B_{i,\g}(y')|<K^{\alpha_i}\max_{j>i}\{
(B^{max}_{j,\g})^{\alpha_i/\alpha_j}\}$$ so that
$$|nB_{i,\g}(y)-nB_{i,\g}(y')| \leq  K^{\alpha_i}D_{M_i}(y,y')^{\alpha_i}+ 2nK^{\alpha_i}\max_{j>i}\{ (B^{max}_{j,\g})^{\alpha_i/\alpha_j}\}$$
for all $n$. In particular
$$|B_{i,\g}(y)-B_{i,\g}(y')| \leq \frac{ K^{\alpha_i}D_{M_i}(y,y')^{\alpha_i}}{n}+ 2K^{\alpha_i}\max_{j>i}\{ (B^{max}_{j,\g})^{\alpha_i/\alpha_j}\}$$
so that as $n\to \infty$
$$|B_{i,\g}(y)-B_{i,\g}(y')|\leq 2K^{\alpha_i}\max_{j>i}\{ (B^{max}_{j,\g})^{\alpha_i/\alpha_j}\}.$$
\end{proof}

\noindent{}{\bf Projection homomorphisms $\mathbf{\tau_j}$.} 
Let $\N=\K_r$ and define $\K_{r-1}=\ker(\tau_r)$ where 
$$\tau_{r}: \K_{r} \to \R^{n_{r}}$$
is given by $\tau_{r}(\k)=B_{r,\k}$.  Now by Lemma \ref{Irineslemma} we know that $\k \in \K_{r-1}$ has the form
$$\k(x_1,\cdots x_r) =(x_1+B_{1,\k}(x_2,\cdots,x_{r-1}), \cdots ,
x_{r-1}+B_{r-1,\k}, x_r)$$
so that it possible to define 
$$\tau_{r-1}: \K_{r-1} \to \R^{n_{r-1}}$$
by $\tau_{r-1}(\k)=B_{r-1,\k}$.
Repeating this argument, we can define $\tau_j$ and $\K_{j-1}=\ker(\tau_j)$ for all $j\leq r$. 
Note also that $\K_j/\K_{j-1} \simeq \tau_j(\K_j)$ is abelian.  We will
show that each $\K_j$ is finitely generated. This will give us a sequence of subgroups each one normal in the next one
$$1 \subseteq  \K_1 \subseteq \cdots \subseteq  \K_{r-1} \subseteq \K_r$$
where the quotients are finitely generated abelian, thus showing that
$\N$ is polycyclic. We will proceed by induction.
%
%
\begin{lemma} $\K_1=\ker(\tau_2)$ is finitely generated abelian.
\end{lemma}
\begin{proof} This follows from Lemma \ref{Irineslemma} and properness
  of the action. By Lemma \ref{Irineslemma} any $\k \in \K_2$ is of the
  form $(x_1+B_{1,\k},\ x_2, \cdots, x_r)$. If $\K_1$ were not finitely
  generated then it would not act properly on $\R^n$.
\end{proof}

\noindent{}Now suppose that $\K_{j}$ is finitely generated for $j<r$.
We will show that $\N=\K_r$ is also finitely generated.
Suppose $\K_r$ is generated by $\{\n_1, \n_2, \cdots \}$ where the first
$d_1$ elements are generators of $\K_1$, the first $d_2$ are
generators of $\K_2$ etc.
Let $\N_i$ be the subgroup generated by $\{\n_1, \cdots, \n_i\}$.  By
Lemma \ref{Kevinslemma} there is some $d$ such that for $i\geq d$ we
have $$l_{i+1}=[\N_{i+1}:\N_i]< \infty.$$ 
We will now state two propositions and explain how they can be used to prove that $\N$ is finitely generated.
%
%
\begin{prop}\label{propA} For any $p> d$ consider the $p$th generator
$\n_p$ and let $l=l_dl_{d+1}\cdots l_{p-1}$, (i.e. the index of $\N_d$
in $\N_p$).  Then there exists $\gamma'$ such that
\begin{enumerate}
\item$ \n_p=\n' \eta,$ where $\eta\in \N_d$.
\item$(\n')^l=\prod_{i=1}^d \n_i^{c_i}$, where $0 \le c_i \le l.$
\item$B_{r,\n'} \le B_{r,\n_1} + ...+
B_{r,\n_d}$
\end{enumerate}
\end{prop}

\begin{prop}\label{propB} Suppose $\n' \in \N$ satisfies conditions (2) and (3) of Proposition  \ref{propA} then there exists an $R$, depending only on $\N$ and $d$, such that  for all $(x_1,\cdots,x_r)$
$$|\n'(x_1,\cdots,x_r)| <R.$$
\end{prop}
Since $\N$ acts properly on $\R^n$ any element of $\N$ that moves points at most distance $R$ must lie in some $\N_{d_R}$ where $d_R$ depends only on $R$. Without loss of generality, we may assume that $d_R>d$. Then, by part (1) of Proposition \ref{propA}, for any $p$ we can write $\n_p=\n' \eta$ where, by Proposition \ref{propB}, $\n'$ moves points at most distance $R$ (in other words $\n' \in \N_{d_R}$) and $\eta \in \N_d \subseteq \N_{d_R}$ so that  $\n_p \in \N_{d_R}$.   This shows that $\N=\N_{d_R}$ and so $\N$ is finitely generated.\\

\begin{proof} \emph{(of Proposition \ref{propA}) }\\
%
%

\noindent{}{\bf Approximating $\n_p$ by an element  $\n_p'\in \N_{d}$.} 
First, compose $\n_p$ with a product of $\n_i$'s so that
$$|B_{r,\n_p}| < \sum_{i=1}^d |B_{r,\n_i}|.$$
Note that $B_{r,\n_p}$ must be in the $\R$ span of the $B_{r,\n_i}$
otherwise the subgroup $\N_d$ would have infinite index in $\N_p$.
Now since $[\N_p:\N_d]=l$ we have that $\n_p^l \in \N_{d}$.  We
will now give an algorithm to define $\n'_p$ as the approximate $l$th
root of $\n_p^l$ in $\N_d$. The following lemma will be useful in our
calculations.
%
%
\begin{lemma}\label{Shufflelemma} If $\k \in \K_{j}$ then for $\n\in \N$
$$B_{i,\n\k}(y)=B_{i,\k\n}(y)=B_{i,\n}(y) \textrm{ if } j<i$$
$$B_{j,\n\k}(y)=B_{j,\k\n}(y)$$
$$B_{j,\k}=B_{j,\n\m} \Rightarrow B_{j,\k\m^{-1}}(y)=B_{j,\n}(y)$$
\end{lemma}
\begin{proof} This follows directly from the definition of the maps.
  If $\k \in \K_{j-1}$ then $B_{i,\k}(y)=0$ for $i>j$ and when $i=j$ we
  have $B_{i,\k}(y)=B_{i,\k}$. The third property holds because
  $\n\m\k^{-1} \in \K_{j-1}$.
\end{proof}

%
%
\noindent{}{\bf Approximating $l$th roots algorithm.}
Set $\m_r=\n_p^l\in \N_{d}$. The algorithm for defining the $l$th
approximate root of $\n_p$ is recursive.  We will define
$\hat{\m_r},\m_r^{err} \in \N_{d}$ and $\m_{r-1} \in \N_{d_{r-1}}$. 
Note that if $j<r$ then by assumption $\N_{d_j}=\K_j$. First, since
$\N_{d}$ is generated by $\{\n_{1},\cdots \n_{d} \}$ we can write
$$B_{r,\m_r}= \sum_{d_{r-1} <  i \leq d} a_i B_{r,\n_i}$$
Let
$$\hat{\m_r}= \prod_{d_{r-1} <  i \leq d} \n_i^{\lfloor \frac{a_i}{l}\rfloor}$$
then
$$B_{r,\hat{\m_r}}=\sum _{d_{r-1} <  i \leq d}\left\lfloor \frac{a_i}{l} \right\rfloor B_{r,\n_i}$$ 
and so
$$B_{r,\m_r} - lB_{r,\hat{\m_r}}=\sum _{d_{r-1} <  i \leq d} c_iB_{r,\n_i} \textrm{ where  } 0\leq c_i< l$$ 
Let
$$\m_r^{err}=\prod _{d_{r-1} <  i \leq d} \n_i^{c_i}$$
Then
$$B_{r,\m_r^{err}}=\sum _{d_{r-1} <  i \leq d} c_iB_{r,\n_i} =B_{r,\m_r} - lB_{r,\hat{\m_r}}=B_{r,\m_r(\hat{\m}_r^{-1})^{l}}=B_{r,(\n_p\hat{\m}_r^{-1})^{l}}$$
This implies that
$$(\n_p\hat{\m}_r^{-1})^{l}(\m_r^{err})^{-1}=\m_{r-1} \in \K_{r-1}$$
Repeat this algorithm to get $\hat{\m}_{r-1}$ and $\m_{r-1}^{err}$.  Now
we know
$$B_{r-1,\m_{r-1}{(\hat{\m}_{r-1}^{l})^{-1}}}=B_{r-1,\m^{err}_{r-1}}$$
so that
$$B_{r-1,(\n_p\hat{\m}_r^{-1})^{l}(\m_r^{err})^{-1}{(\hat{\m}_{r-1}^{l})^{-1}}}=B_{r-1,\m^{err}_{r-1}}$$
but by the second equality of Lemma \ref{Shufflelemma} we have
$$B_{r-1,(\n_p(\hat{\m}_r\hat{\m}_{r-1})^{-1})^{l}(\m_r^{err})^{-1}}=B_{r-1,\m^{err}_{r-1}}$$So again 
$$(\n_p(\hat{\m_r} \hat{\m}_{r-1})^{-1})^{l}(\m_r^{err})^{-1}(\m^{err}_{r-1})^{-1}=\m_{r-2}\in \K_{r-2}$$
Repeating this process we get $\hat{\m}_r, \hat{\m}_{r-1}, \cdots,
\hat{\m}_1$ so that
$$(\n_p(\hat{\m_r}\cdots \hat{\m}_{1})^{-1})^{l}(\m_r^{err})^{-1}\cdots(\m^{err}_{1})^{-1}=Id$$
Let $\eta=\hat{\m}_{1} \cdots \hat{\m_r}$ and
$\n'=\n_p(\eta)^{-1}$ then $\n_p= \n' \eta$ and
$$(\n')^{l}=\m^{err}_{1}\cdots \m_r^{err}=\prod_{i=1}^{d_r} \n_i^{c_i}$$
as promised in (1) and (2). \\

\noindent{}To show (3) we need to simply note that  $B_{r,\n \m}= B_{r,\n} + B_{r,\m}$ for any $\n, \m$.
Applying this to $(\n') ^l$ we get
$$l B_{r,\n'} = c_1 B_{r,\n_1} + \cdots + c_d B_{r,\n_d} \leq  l(B_{r,\n_1} + \cdots + B_{r,\n_d}).$$
\end{proof}

\begin{proof} \emph{(of Proposition \ref{propB})}

\noindent{}We will show that for each $1\leq i\leq r$
$$B_{i,\n'}(y_i)<R_i$$
so that we can take $R=R_1+R_2 + \cdots + R_r$.  The following lemma
will give us an estimate on the maximum of $B_{i,\n'}(y_i)$.

\begin{lemma}\label{Estimationlemma} Recall that
  $B^{max}_{i,\n}=\sup_y\{B_{i,\n}(y)\}$ and $\e_{i,\n}=\max_{j>i}\{
  2K^{\alpha_i}(B^{max}_{j,\n})^{\alpha_i/\alpha_j}\}$. The following
  inequalities hold:
$$B^{max}_{i,\n\m}\leq B^{max}_{i,\n} + B^{max}_{i,\m}$$
$$lB^{max}_{i,\n}\leq B^{max}_{i,\n^l}+l\e_{i,\n}$$
\end{lemma}
\begin{proof}
  The first inequality follows from the fact that
$$B_{i,\n\m}(y)=B_{i,\n}(y')+B_{i,\m}(y).$$
The second inequality follows from the above equality and from the
following estimate given in Lemma \ref{Irineslemma}:
$$|B_{i,\n}(y) - B_{i,\n}(y')| < \e_{i,\n}$$
in other words
$$|B_{i,\n^l}(y)|\geq |lB_{i,\n}(y)| -  l\e_{i,\n}.$$
\end{proof}
Now by Lemma \ref{Estimationlemma} we know that \bea
lB^{max}_{i,\n'} &\leq& B^{max}_{i,(\n')^l}+l\e_{i,\n'}\\
&\leq& c_1B^{max}_{i,\n_1}+ \cdots +c_rB^{max}_{i,\n_r}+l\e_{i,\n'}\\
&\leq& l(B^{max}_{i,\n_1}+ \cdots +B^{max}_{i,\n_r})+l\e_{i,\n'} \eea
{\bf Claim:} $\e_{i,\n'}$ does not depend on $\n'$.  Note that if
$B^{max}_{j,\n'}$ does not depend on $\n'$ for $j>i$ then
$\e_{i,\n'}$ does not depend on $\n'$. By part (3) of Proposition \ref{propB}
we know that 
 $$B^{max}_{r,\n'}<B^{max}_{r,\n_1}+ \cdots +B^{max}_{r,\n_r}$$
 (since for $j=r$ we have $B^{max}_{r,\n}=B_{r,\n}$).
 Using this and the above estimate we get that $B^{max}_{r-1,\n'}$ does not depend on $\n'$. Proceeding inductively proves the claim and so we can write  $\e_{i}=\e_{i,\n'}$ for any $\n'$ and  take $R_i=B^{max}_{i,\n_1}+ \cdots +B^{max}_{i,\n_r}+\e_{i}$.
 \end{proof}

\subsection{Showing $\Gamma$ is virtually a lattice in
  $\R\ltimes_{M'} \R^n$}\label{RbyRn}
At this point we know that any finitely generated group $\Gamma$ that
is quasi-isometric to $G_M$ must be virtually polycyclic. By a theorem
of Mostow \cite{Mos}, every polycyclic group $\Gamma$ is virtually a
uniform lattice in a simply connected solvable Lie group $\Lie$. We would like to conclude
that
$\Lie \simeq \R\ltimes_{M'} \R^n$ where $M'$ is a matrix with the same absolute Jordan form as $M^\alpha$ for some $\alpha\in \R$.\\

\noindent{\bf Nilradical and exponential radical.} From \cite{Aus} we
know that any connected simply connected solvable Lie group $\Lie$ has the form
$$1 \to \Nie \to \Lie \to \R^s \to 1$$
where $\Nie$ is the \emph{nilradical} of $\Lie$, the unique maximal
connected normal nilpotent subgroup. Related to the nilradical is
$R_{exp}(\Lie)$ the \emph{exponential radical}: the set of exponentially
distorted elements of $\Lie$. From \cite{O} and \cite{Gu}, we know that $R_{exp}(\Lie)
\subset \Nie$.  Furthermore, in \cite{Cor} Cornulier shows that for a simply connected
solvable Lie group $L$ the dimension 
$dim{\Lie/R_{exp}(\Lie)}$ is a quasi-isometry invariant. We can also compute that if $M$ has all
eigenvalues off the unit circle then $R_{exp}(\R \ltimes_M
\R^n)\simeq \R^n$. Putting these results together we get
$$dim(\Lie/R_{exp}(\Lie))=1.$$
In particular,  we must have $\Lie\simeq \R \ltimes \Nie$ and $R_{exp}(\Lie)=\Nie$.\\

\noindent{\bf Representation into $QI(G_M)$.} Now $\Lie \simeq \R \ltimes
\Nie$ is a locally compact topological group (virtually) containing
$\Gamma$ as a cocompact lattice.  By Furman's construction 3.2 in
\cite{F} we have a representation
$$\pi' : \Lie \to QI(\Gamma).$$
If $\Gamma$ is quasi-isometric to $G_M$ via $\rho: \Gamma \to G_M$
then this representation extends to a map
$$\pi=Ad_\rho \circ \pi': \Lie \to QI(G_M)$$ 
where $Ad_\rho(q)= \rho q \bar{\rho}$ and $\bar{\rho}$ is a coarse
inverse of $\rho$.
Each $\pi_h$ is a $(M,A)$ quasi-isometry where $M$ and $A$ are constants that do not depend on $h$. We can assume that $\rho$ is injective and that $\bar{\rho}\rho\vert_\Gamma = Id$ and that $\rho(e_\Gamma)=e_{G_M}$.\\

\noindent{\bf Showing Continuity.} 
Since $QI(G_M) \simeq QSim_{D_{M_l}}(\R^{n_l}) \times
QSim_{D_{M_u}}(\R^{n_u})$, we can compose $\rho$ with the two obvious
projections to get two maps
 $$\phi_l: \Lie \to QSim_{D_{M_l}}(\R^{n_l}) \subset Homeo(\partial_l G_M)$$
 and
 $$\phi_u: \Lie \to QSim_{D_{M_u}}(\R^{n_u}) \subset Homeo(\partial_u G_M).$$ 

\begin{prop} The maps $\phi_l$ and $\phi_u$ defined above are continuous 
with respect to the topology of uniform convergence.
\end{prop}
\begin{proof}
We will follow the argument found in \cite{F}.  For any $v\in \Lie$ define a quasi-isometry $q_v$ of $\Gamma$ as follows. 
Chose some open subset $E \subset \Lie$  with compact closure, such that $\Lie=\cup_{\gamma \in \Gamma} \gamma E$ and fix a function
$p: \Lie \to \Gamma$ satisfying $v \in p(v) E.$  Define $q_v(\gamma):= p( v \gamma )$. 
 Now let $\B_k$ be a ball of radius $k$ in $G_M$ and consider
 $F_k=\bar{\rho}(\B_k)$. This is a finite subset of $\Gamma$ so by
 Lemma 3.3 (d) in [F] there is a neighborhood of the identity $V_k
 \subset \Lie$ such that  $d(q_v(\g), \g) \leq B$ for all $\g
 \in F_k$ and $v \in V_k$. Furthermore, the constant $B$ does not depend on
 $F_k$. In particular, for all $\g \in F_k$ and $v \in V_k$
$$d(\rho (q_v (\g)), \rho(\g))\leq KB + C.$$
So now for all $x \in \B_r$
 $$d(\rho q_v \bar{\rho}(x), x)\leq d(\rho (q_v (\g)), \rho(\g))+ d(\rho(\g),x) \leq KB + C + C.$$

 Set $B'=KB+C+C$.
 Let $L_e \subset \F_l$ be the leaf of the foliation $\F_l$ (the
 foliation defined by isometrically embedded copies of the negatively
 curved homogeneous space $G_{M_l}$) containing the identity of $G_M$. In coordinates $L_e$ is simply given by $\{ (x, y , t)  \mid y=(0,\ldots 0) \}$. 
 Since $L_e$ is a negatively curved space, we can consider the space
 of all unit speed geodesic rays emanating from $e$. This space is
 equivalent to the visual boundary $\partial L_e$ which in turn is
 equivalent to $\partial_l G_M \cup {\infty}$.

 Now let $\point\in \partial_l G_M$ and let $\geod(t)$ be a geodesic ray
 emanating from $e$ that represents $\point$.  For any $v\in \Lie$ the map
 $\rho q_v \bar{\rho} $ is a quasi-isometry of $G_M$ and so induces a
 quasi-isometry of $G_{M_l}$: the leaf $L_e\subset G_M$ is mapped to
 witin a bounded distance of $L'\subset G_M$ where $L'$ and $L_e$ are
 both isomorphic to $G_{M_l}$ so we can identify $L_e$ and $L'$ via
 this isomorphism. We will treat $\rho q_v \bar{\rho}$ as a
 quasi-isometry of $G_{M_l}$.  (Note that really what is going on is
 that we are projecting $L'$ onto $L_e$ and this projection is
 distance non increasing. That is we are considering $q_v'=\pi_{L_e}
 \rho q_v \bar{\rho} \vert_{L_e}$ where $\pi_{L_e}: G_M \to L_e$ is
 projection onto $L_e$.)  Now $\geod'(t)=q_v'\geod(t)$ is a $(K',C')$
 quasi-geodesic ray representing $v(\point)\in \partial_l G_M$ where
 $K',C'$ depend only on $K,C,B$. By the ``Morse Lemma'' there exists a
 geodesic uniformly bounded from $\geod'$. Let $\geod_{v(\point)}$ denote this
 geodesic and $C''$ the uniform bound.
 Then for all $t\in [0,k]$, since $\geod_\point(t) \in B_k$ for $t \in
 [0,k]$, we have
$$d(\geod_{v(\point)}(t), \geod_\point(t)) \leq C''+ d(q_v'\geod_\point(t),\geod_\point(t)) \leq C''+ B'.$$
Therefore if we define a uniform structure $U_k \subset \partial_l G_M
\times \partial_l G_M$ to be all pairs of geodesic rays $(\geod_1,\geod_2)$
such that
$$d(\geod_1(t),\geod_2(t))\leq C''+B'\quad \textrm{ for } \quad t \in[0,k]$$
then for all $v \in V_k$ we have $(v(\point),\point) \in U_k$. This shows
that
$$\phi_l: \Lie \to Homeo(\partial_l G_M)$$
is continuous with respect to uniform convergence in $Homeo(\partial_l
G_M)$.  Similarily
 $$\phi_u: \Lie \to Homeo(\partial_u G_M)$$ 
 is also continuous.
 \end{proof}

\noindent{\bf Showing continuity of the height function.}
Now since both $\phi_l$ and $\phi_u$ are continuous then
$\pi(\Lie)=(\phi_l\times\phi_u)(\Lie)$ is a separable subgroup of
$QSim_{D_{M_l}}(\R^{n_l}) \times QSim_{D_{M_u}}(\R^{n_u})$. By
Propositions \ref{twoconj} and \ref{stretchinverse} we know that we
can conjugate $\pi(\Lie)$ into $AIsom(G_M)$. Let $\phi: \Lie \to AIsom(G_M)$
be $\pi$ composed with this conjugation. Recall that on $AIsom(G_M)$
there is a well defined height homormorphism $h$ sending each almost
isometry to its height translation. We will show that $h$ is
continuous in order to conclude that the composition $\tilde{h}=h\circ
\phi$ is also continuous.

\begin{lemma} The height homomorphism
$$h:AIsom(G_M) \to \R$$ 
is continuous.
\end{lemma}
\begin{proof}
  Let $q_i$ be a sequence of quasi-isometries in $AIsom(G_M)$ such
  that $q_i \to id$.  Note that from the structure of $AIsom(G_M)$
  maps we can see that $q_i$ induces a height respecting isometry of
  $\mathbb{H}^{k}$ for some $k$. If a sequence of height respecting
  isometries converge to the identity then their height translation
  must be going to zero. Therefore, $h(q_i) \to 0$.
\end{proof}

\noindent{}Next we show that $\{0\}\times \Nie \subset \ker{\tilde{h}}$.
We use the fact that $\{0\}\times \Nie$ is exponentially distorted in $\R
\ltimes \Nie$ and the following lemma:

\begin{lemma} Let $\tilde{h}:\Lie \to \R$ be a continuous homomorphism from a
  locally compact, compactly generated group $\Lie$ to $\R$. Then there
  exists a constant $K$ such that
$$|\tilde{h}(\g)| \leq K|\g|_\Lie$$ 
where $|\cdot |_\Lie$ is the induced length in $\Lie$.
\end{lemma}

\noindent Now let $\gamma\in \R \ltimes \Nie$ be an exponentially
distorted element and suppose $\tilde{h}(\gamma)=c$. Then, for any $n$
$$cn=|\tilde{h}( \gamma^n)| \leq K|\gamma^n| \leq K \log{(n+1)}.$$
Therefore $\gamma \in \ker{\tilde{h}}$ and so $\{0\} \times \Nie \subset
\ker{\tilde{h}}$.

Next, we want to show that $\ker{\tilde{h}}= \{0\}\times \Nie$. Suppose that
$(r,e) \in \ker{\tilde{h}}$ where $r\neq 0$. Then, by continuity, all of $\R
\ltimes \Nie$ is contained in $\ker{\tilde{h}}$ but this is impossible since $\tilde{h}$
must be onto, (otherwise the quasi-action would not be cobounded).

So now we know that $\Nie$ quasi-acts on $G_M$ by quasi-isometries that
fix height.  Let $x \in G_M$. Then $\gamma \mapsto \pi(\gamma)(x)$ is
a quasi-isometry from $\R \ltimes \Nie \to \R \ltimes \R^n$. If we restrict
this map to $\Nie$ we get a quasi-isometric embedding $\Nie \to \R^n$.
Furthermore, since in our case cohomological dimension is a quasi-isometry
invariant \cite{Ge}, we know that $\Nie$ must be an $n$ dimensional Lie
group. Now we can conclude, by Theorem 7.6 in \cite{FM3}, that this
quasi-isometric embedding must be coarsely onto as well. Therefore, $\Nie$ is
quasi-isometric to $\R^n$ and the only nilpotent Lie group that can be
quasi-isometric to $\R^n$ is $\R^n$ itself. Therefore, $\R \ltimes \Nie =
\R \ltimes_S \R^n$.

Finally, note that $\gamma \to \pi(\gamma)(x)$ is actually a height respecting quasi-isometry. Then by Theorem 5.11 from \cite{FM3}, we can conclude that for some $\alpha$, the matrix $M^\alpha$ has the same absolute Jordan form as $S$.\\

\end{document}